\documentclass{amsart}
\usepackage{latexsym}
\usepackage{amsmath}
\usepackage{amsthm}
\usepackage{amssymb}
\usepackage{hyperref}
\usepackage{color}
\usepackage{enumitem}

\newcommand{\vip}{\vskip0.15cm}

\newcommand{\indiq}{1\!\! 1}
\newcommand{\e}{{\varepsilon}}

\newcommand{\E}{{\mathbb{E}}}

\newcommand{\be}{{\bf e}}
\newcommand{\bv}{{\bf v}}
\newcommand{\bV}{{\bf V}}
\newcommand{\bw}{{\bf w}}
\newcommand{\bW}{{\bf W}}
\newcommand{\bY}{{\bf Y}}

\newcommand{\abs}[1]{\lvert#1\rvert}

\newcommand{\cH}{{{\mathcal H}}}

\newcommand{\cW}{{{\mathcal W}}}
\newcommand{\cL}{{{\mathcal L}}}
\newcommand{\cA}{{{\mathcal A}}}
\newcommand{\cP}{{{\mathcal P}}}

\newcommand{\rr}{{\mathbb{R}}}
\newcommand{\Sp}{{\mathbb{S}}}
\newcommand{\rd}{{\mathbb{R}^3}}

\newcommand{\intot}{\int_0^t}
\newcommand{\intrd}{\int_{\rd}}

\newcommand{\tv}{{\tilde v}}
\newcommand{\tx}{{\tilde x}}
\newcommand{\tc}{{\tilde c}}

\newcommand{\tB}{{\tilde B}}
\newcommand{\tW}{{\tilde W}}
\newcommand{\tM}{{\tilde M}}

\newcommand{\sm}{{s-}}

\newcommand{\beqn}{\begin{equation}}
\newcommand{\eeqn}{\end{equation}}
\newcommand{\bear}{\begin{eqnarray}}
\newcommand{\eear}{\end{eqnarray}}
\newcommand{\bean}{\begin{eqnarray*}}
\newcommand{\eean}{\end{eqnarray*}}

\def\beq{\begin{equation}}\def\eeq{\end{equation}}
\def\benu{\begin{enumerate}} \def\eenu{\end{enumerate}}

\newenvironment{preuve}{\vip\noindent {\it Proof}}{\hfill$\square$\vip}

\newtheorem{theo}{\indent Theorem}[section]
\newtheorem{prop}[theo]{\indent Proposition}
\newtheorem{rem}[theo]{\indent Remark}
\newtheorem{lem}[theo]{\indent Lemma}
\newtheorem{defin}[theo]{\indent Definition}

\begin{document}

\title[The Kac particle system]
{Rate of convergence of the Kac particle system for the Boltzmann equation with hard potentials}

\author{Chenguang Liu}
\address{Delft Institute of Applied Mathematics, EEMCS, TU Delft, 2628 Delft, The Netherlands}
\email{C.Liu-13@tudelft.nl}

\author{Liping Xu}
\address{Beihang University, School of Mathematical Sciences, No 37 XueYuan Road, Haidian, Beijing }
\email{xuliping.p6@gmail.com }

\author{An Zhang}
\address{Beihang University, School of Mathematical Sciences, No 37 XueYuan Road, Haidian, Beijing }
\email{anzhang@pku.edu.cn }



\begin{abstract} 
In this paper, we prove that the Kac stochastic particle system converges to the weak solution of  the spatially homogeneous Boltzmann equation for hard potentials and hard spheres.  We give, under the initial data with finite exponential moment assumption, an explicit rate of propagation of chaos in squared Wasserstein distance with quadratic cost by using a double coupling technique.
\end{abstract}

\maketitle

\section{Introduction and main results} 
\setcounter{equation}{0}

\subsection{The Boltzmann equation}
We consider a 3-dimensional spatially homogeneous  Boltzmann equation, which depicts the density $f(t,v)$ of particles in a dilute gas, moving with velocity $v\in\rd$ at time $t\geq 0$. The density $f_t(v)$ satisfies
\begin{eqnarray} \label{be}
\partial_t f_t(v) = \frac 1 2\intrd dv_* \int_{\Sp^{2}} d\sigma B(|v-v_*|,\theta)
\big[f_t(v')f_t(v'_*) -f_t(v)f_t(v_*)\big],
\end{eqnarray}
where 
\begin{equation}\label{vprimeetc}
v'=v'(v,v_*,\sigma)=\frac{v+v_*}{2} + \frac{|v-v_*|}{2}\sigma, \quad 
v'_*=v'_*(v,v_*,\sigma)=\frac{v+v_*}{2} -\frac{|v-v_*|}{2}\sigma
\end{equation}
and $\theta$ is the \emph{deviation angle} given by
$\cos \theta = \frac{(v-v_*)}{|v-v_*|} \cdot \sigma$. Due to Galilean invariance,  
 the {\em collision kernel} is assumed to be $B(|v-v_*|,\theta)\geq 0$ depending on the type of interaction between the particles and giving the rate at which the pair of particle collide. It is determined by both $|v-v_*|$ and the cosine of the deviation angle $\theta$. See Cercignani \cite{MR1313028},
Desvillettes \cite{MR1881103}, Villani \cite{MR1942465} and Alexandre \cite{MR2556715} for physical explanations and mathematical reviews on this equation. It's well known that the equation is given under five assumptions, thus the conservation of mass, momentum and kinetic energy hold for reasonable solutions and we may assume without loss of generality that $\int_{\rd} f_0(v) dv=1$.

\subsection{Assumptions}
We will consider the classical physical example of collision kernel which is given by inverse power laws interactions: when particles interact by pairs due to a repulsive force proportional to $1/r^s$ for some $s>2$, the following assumption \eqref{cs} holds with $\gamma=(s-5)/(s-1)$ and $\nu=2/(s-1)$.  In this case, the collision kernel can be computed implicitly, that is,  there is a measurable function $\beta:(0,\pi]\rightarrow\rr_+$ such that
\beq\label{cs}
 B(|v-v_*|,\theta)\sin\theta={|v-v_*|}^{\gamma}\beta(\theta),\ \hbox{and}\  \forall~\theta\in[\pi/2,\pi],~\beta(\theta)=0,
\eeq
and either 
\begin{equation}\label{c2hs}
\forall \theta\in(0,\pi/2),  \; \beta(\theta)=1
\end{equation}
or
\begin{equation}\label{c2}
\exists \; \nu\in (0,1), \; \exists \; 0<c_0<c_1,\; \forall \theta\in(0,\pi/2),  \;
c_0 \theta^{-1-\nu} \leq \beta(\theta) \leq c_1 \theta^{-1-\nu}.
\end{equation}
In addition, we will also assume that 
\begin{equation}\label{c4}
\beta(\theta)=b(\cos\theta)\quad \hbox{with $b$ non-decreasing, convex and $C^1$ on $[0,1)$}.
\end{equation}
This  additional condition is required since the exponential moments are considered in the sequel.
The assumption $\beta=0$ on $[\pi/2,\pi]$  is not a restriction and can be obtained by  symmetry as noted in the introduction of \cite{MR1765272}. Here we will focus on the case of hard potentials and hard spheres, i.e. $\gamma\in(0,1]$.

\subsection{Some notations}
Let  us denote by $\cP(\rd)$  the set of all probability measures on $\rd$.  When $f\in\cP(\rd)$ has a density, we also denote this density  by $f$.
For $q>0$, we denote
\[\cP_q(\rd)=\{f\in\cP(\rd): m_q(f)<\infty\} \quad\text{with}~~m_q(f):=\int_{\rd}|v|^q f(dv).\]
We now introduce, for $\theta\in(0,\pi/2)$ and $z\in[0,\infty)$,

\beq\label{defH}
H(\theta)=\int_\theta^{\pi/2}\beta(x)dx\quad\text{and}\quad G(z)=H^{-1}(z).
\eeq
Under \eqref{c2}, it is clear that $H$ is a continuous decreasing function valued in  $[0,\infty)$, so  it has an inverse function $G:[0,\infty)\mapsto(0,\pi/2)$ defined by $G(H(\theta))=\theta$  and $H(G(z))=z$. Furthermore, it is easy to verify  that there exist some constants $0<c_2<c_3$ such that for all $z>0$,
\beq\label{eG}
c_2(1+z)^{-1/\nu}\le G(z)\le c_3(1+z)^{-1/\nu},
\eeq
and we know from \cite[Lemma 1.1]{MR2398952} that there exists a constant $c_4>0$ such that for all $x,y\in\rr_+$,
\beq\label{c3}
\int_0^{\infty}(G(z/x)-G(z/y))^2dz\le c_4\frac{(x-y)^2}{x+y}.
\eeq
Under \eqref{c2hs}, it's clear that $G(z)=\max{(\pi/2-z,0)}$, and a direct computation shows that $G(z)$ satisfying \eqref{c3} as well.
\vip
Let us now introduce the Wasserstein distance with quadratic cost on $\cP_2(\rd)$.
For $g,\tilde{g}\in\cP_2(\rd)$, let $\cH(g,\tilde{g})$ be the set of probability measures on $\rd\times\rd$ with first marginal $g$ and second marginal $\tilde{g}$. We then set
\[ \cW_2(g,\tilde{g}) =  \inf\left\{\Big(\int_{\rd\times\rd}|v-\tilde{v}|^2R(dv,d\tilde{v})\Big)^{1/2},~~R\in\cH(g,\tilde{g})\right\}. \]
For more details on this distance, one can see \cite[Chapter 2]{MR1964483}.

\subsection{Weak solutions}

We  now introduce a suitable spherical parameterization of \eqref{vprimeetc} as in \cite{MR1885616}.
For each $X\in \rd$, we consider  vectors $I(X),J(X)\in\rd$ such that
$(\frac{X}{|X|},\frac{I(X)}{|X|},\frac{J(X)}{|X|})$ 
is a direct orthonormal basis of $\rd$. 
Then for $X,v,v_*\in \rd$, for $\theta \in (0,\pi/2)$ and $\varphi\in[0,2\pi)$,
we set
\begin{equation}\label{dfvprime}
\left\{
\begin{array}{l}
\Gamma(X,\varphi):=(\cos \varphi) I(X) + (\sin \varphi)J(X), \\\\
a(v,v_*,\theta,\varphi):= - \displaystyle\frac{1-\cos\theta}{2} (v-v_*)
+ \frac{\sin\theta}{2}\Gamma(v-v_*,\varphi),\\
v'(v,v_*,\theta,\varphi):=v+a(v,v_*,\theta,\varphi),
\end{array}
\right.
\end{equation}
then we  write $\sigma\in \Sp^2$ as $\sigma=\frac{v-v_*}{|v-v_*|}\cos\theta
+ \frac{I(v-v_*)}{|v-v_*|}\sin\theta \cos\varphi+\frac{J(v-v_*)}{|v-v_*|}\sin\theta
\sin\varphi,$ and  observe at once that $\Gamma(X,\varphi)$ is orthogonal to $X$ and has the same norm as $X$,
from which it is easy to check that
\beq\label{num}
|a(v,v_*,\theta,\varphi)|=\sqrt{\frac{1-\cos\theta}{2}}|v-v_*|.
\eeq

Let us give the typical definition of  weak solutions, i.e. measure-valued solutions, to \eqref{be}.

\begin{defin}\label{dfsol}
Assume \eqref{cs} and \eqref{c2hs} or \eqref{c2}.
A family $(f_t)_{t\geq 0} \in C([0,\infty),\cP_2(\rd))$ is called a weak solution to \eqref{be} if it satisfies 
\begin{itemize}
\item For all $t\geq 0,$
\begin{equation}\label{cons}
 \intrd v f_t(dv)= \intrd v f_0(dv) \quad\hbox{and}\quad
\quad \intrd |v|^2 f_t(dv)= \intrd |v|^2 f_0(dv).
\end{equation}
\item For any $\phi:\rd\mapsto \rr$  bounded and Lipschitz-continuous, any $t\in [0,T]$,
\begin{equation}\label{wbe}
\intrd \phi(v)\, f_t(dv) =  \intrd \phi(v)\, f_0(dv)
+\intot \intrd \intrd \cA\phi(v,v_*) f_s(dv_*)   f_s(dv) ds
\end{equation}
where
\begin{equation}\label{afini}
\cA\phi (v,v_*) = |v-v_*|^\gamma \, \int_0^{\pi/2}
\beta(\theta)d\theta 
\int_0^{2\pi} d\varphi
\left[\phi(v+a(v,v_*,\theta,\varphi))-\phi(v) \right].
\end{equation}
\end{itemize}
\end{defin}

It is easy to get $|\cA\phi(v,v_*)|\leq C_\phi |v-v_*|^{1+\gamma} \leq C_\phi  (1+|v-v_*|^2)$ from that $|a(v,v_*,\theta,\varphi)| \leq C  \theta |v-v_*|$ and that $\int_0^{\pi/2}
\theta \beta(\theta)d\theta<\infty$,
so that \eqref{wbe} is well-defined.
\vip

Once we have stated the precise definition of weak solutions in the former paragraph, we now provide the  known well-posedness results for the Boltzmann equation, as well as some properties  of solutions that we will need. 

\begin{theo}\label{wp}
Assume \eqref{cs}, \eqref{c4}  and \eqref{c2hs} or \eqref{c2}. Let $f_0\in \cP_2(\rd)$.  For $\gamma\in (0,1]$, we assume additionally 
\begin{equation}\label{c5}
\exists \; p_0\in(\gamma,2),\;\;
\intrd e^{|v|^{p_0}}f_0(dv)<\infty.
\end{equation}
There is a unique weak solution $(f_t)_{t\geq 0}\in C([0,\infty),\cP_2(\rd))$ to \eqref{be} such that
\begin{equation}\label{momex}
\forall \;q\in(0,p_0), \quad \sup_{[0,\infty)}\intrd e^{|v|^q}f_t(dv)<\infty.
\end{equation}
Under \eqref{c2} and if $f_0$ is not a Dirac mass, then $f_t$ has a density for all $t>0$. Under \eqref{c2hs} and if $f_0$ has a density, then $f_t$ has a density for all $t>0$.
\end{theo}

The well-posedness for hard potentials can be  found  in \cite{MR2511651,MR2525118}  and    for hard spheres, for instance see \cite{MR339666,MR1697562,MR1716814,MR2871802,MR2728840}. There are substantial literature on the propagation of exponential moments for hard potentials and hard spheres. The seminal work was given by Bobylev \cite{MR1478067} for the case of short ranged interactions which was later significantly improved in \cite{MR2096050,MR2511651,MR2871802}, and recently Fournier \cite{MR4315665} obtained a stronger result than what happens under the cutoff case, i.e. $\int_0^{\pi/2}\beta(\theta)d\theta<\infty$.
Finally, the existence of a density for $f_t$ has been proved in
\cite{MR3313757} (under \eqref{c2} and when $f_0$ is not a Dirac mass and belongs to $\cP_4(\rd)$),
in \cite{MR1697562} (under \eqref{c2hs} when $f_0$ has a density) and  e.g. in  \cite{MR339665} which is classical by using the monotonicity of the entropy under the assumption that  $f_0$ has a finite entropy. By the way,  the global existence, uniqueness, and large-time behavior for solutions were established  for the Vlasov-Poisson-Boltzmann /Landau system in \cite{MR4688693} and  the propagation of the exponential moments was derived for  the inhomogeneous Boltzmann equation with soft potentials, see \cite{MR4704643}.

\subsection{The particle system}
Let us now  recall  the Kac particle system introduced by Kac in \cite{MR84985} to justify the spatially homogeneous Boltzmann equation. It is
the $(\rd)^N$-valued Markov process with infinitesimal generator $\cL_N$ defined as follows: for any bounded Lipschitz test function
$\phi:(\rd)^N\mapsto \rr$ sufficiently regular and $\bv=(v_1,\dots,v_N) \in (\rd)^N$, by
\begin{align}\label{kacg}
\cL_N \phi(\bv)= \frac 1 {2(N-1)} \sum_{i \ne j} \int_{\Sp^2} 
[\phi(\bv + (v'(v_i,v_j,\sigma)-v_i)\be_i
&+(v'_*(v_i,v_j,\sigma)-v_j)\be_j) 
- \phi(\bv)]\\
 &\times B(|v_i-v_j|,\theta)d\sigma\notag,
\end{align}
where   $v\be_i=(0,\dots,0,v,0,\dots,0)\in(\rd)^N$ with $v$ at the $i$-th place for $v\in\rr^3$.
\vip

Let's briefly interpret  the intrinsic interaction between the particles in  the system. Roughly speaking, we consider $N$ particles with velocities ${\bv}=(v_1,...,v_N)\in(\rd)^N$ in the system. Any pair of particles with velocities $(v_i,v_j)$ interact in a way of interactions described  by the Boltzmann equation, i.e. interacting with deviation angle $\theta$
at rate $B(|v_i-v_j|, \theta)/2(N-1)$ for each $\sigma\in\Sp^2$. Then they change their velocities simultaneously  from $v_i$ to
$v'(v_i,v_j,\sigma)$ and $v_j$ changes to $v'_*(v_i,v_j,\sigma)$ given by \eqref{vprimeetc}. In this paper, we consider the difficult non-angular cutoff collision kernel, i.e. the function $\beta(\theta)$ has a non-integrable singularity when $\theta$ is close to $0$. In this case, the  particles jump infinitely many times with a very small deviation angle.

\subsection{Main result}

The  following is our main result relating to the Kac particle system.

\begin{theo}\label{mr} 
Let $B$ be a collision kernel satisfying \eqref{cs}, \eqref{c4} and \eqref{c2hs} or \eqref{c2}. Let  $f_0\in\cP_2(\rd)$ not be a Dirac mass and satisfying additionally  \eqref{c5}.
Consider the unique weak solution $(f_t)_{t\geq 0}$ to \eqref{be} defined in Theorem \ref{wp}
and, for each $N\geq 1$, 
the unique Markov process $(V^{i}_t)_{i=1,\dots,N,t\geq 0}$  defined in Proposition \ref{kpst}.
Let $\mu^{N}_t:= N^{-1}\sum_{i=1}^N \delta_{V^{i}_t}$. Then
\vip
(i) Hard potentials. Assume that $\gamma\in(0,1)$ and \eqref{c2}.   For all $\e\in(0,1/2)$, all $T\geq0$, there is 
$C_{\e,T}>0$ such that for all $N\geq 1$, 
\begin{align}\label{fc4}
\sup_{[0,T]}\E[\cW_2^2(\mu^{N}_t,f_t)] \leq& C_{\e,T} (N^{-1/3} )^{1-2\e}. 
\end{align}
\vip
(ii) Hard spheres. Assume that $\gamma=1$,  \eqref{c2hs} and that $f_0$ has a density.
 For all $\e\in(0,\nu/4)$, all $T\geq0$, there is 
$C_{\e,T}>0$ such that for all $N\geq 1$,
\begin{align}\label{fc6}
\sup_{[0,T]}\E[\cW_2^2(\mu^{N}_t,f_t)] \leq C_{\e ,T}(N^{-1/3})^{1-4\e/\nu}.
\end{align}
\end{theo}

\subsection{Known results, approaches  and comments.}
We will first give a brief but not exhaustive overview of the known works on the propagation of chaos for the Boltzmann equation. The propagation of chaos for the Boltzmann  equation  is  pioneered by Kac in \cite{MR84985}   without convergence rate, then by \cite{MR334788, MR224348} for bounded collision kernel and \cite{MR753814} for unbounded collision kernel in hard spheres case without rate as well. Graham and M\'{e}l\'{e}ard \cite{MR1428502}, using Tanaka's coupling method, obtained one of the first quantitative results for the Boltzmann equation  with cutoff for  Maxwell molecules. Afterwards,  important progresses have been made  in the remarkable work of Mischler and Mouhot \cite{MR3069113}, where strong uniformly in time results but not sharp  were obtained relying on a purely abstract analytic method. In \cite{MR3069113}, the rate of chaos is given by  $N^{-1/(6+\delta)}$ for any $\delta>0$ for  Maxwell molecules, and the rate is  of $(\log N)^{-r}$ for some $r>0$ for the hard sphere case. Then this  rate for Maxwell case was  greatly improved  by Cortez and Fontbona \cite{MR3769742}   to an almost {\it optimal} rate of $N^{-1/3}$. Recently, Heydecker \cite{MR4419606} gave a rate of $(\log N)^{1-1/\nu}$ with  $\nu\in(0,1)$ for Maxwell, hard potentials and hard spheres under a weaker initial condition, i.e. the initial data has a finite $p$-th order moments for some large explicitable $p$.  

\vip

In this paper, we also consider the propagation of chaos for the non-cutoff hard potentials and  hard spheres but in a different way from that in \cite{MR4419606}, where, to our knowledge, the only result for these cases is given. Let's now give a  concise comparison of our work with \cite{MR4419606}. First of all,  both of the work used the Tanaka's trick \cite{MR512334} and followed the idea of  \cite{MR3456347}, but we used the the Wasserstein distance $\cW_{2}$ while the latter adopted  an optimal cost that is equivalent to the Wasserstein distance $\cW_{p+2}$. This difference inevitably leads to  the different investigation of estimation.  Secondly, we considered the different initial condition. Heydecker \cite{MR4419606} obtained a far from optimal  rate of $(\log N)^{1-1/\nu}$ with  $\nu\in(0,1)$ (equivalent   to that of \cite{MR1720101}), when the initial condition $f_0$ has a finite $p$-th order moments for some large $p$.  However, we  refine  this rate to a {\it close to optimal rate} of $N^{-1/3}$ but with a stronger initial condition $f_0$ that has a finite exponential moment.  Finally, we both exploited  the coupling strategy which is widely used since Sznitman \cite{MR1108185} for studying rate of chaos for the McKean-Vlasov model. Our proof is mainly based on  a new version of  coupling method, called second coupling method,  introduced in \cite{MR3769742,MR3621429}:  we couple the $N$-particle system with effective binary interactions with a family of $N$ non-independent Boltzmann processes driven by the same Poisson random measure, in such a way, the particles in the system behave like the original one determined by the Boltzmann equation as much as possible. In the second step, we show that these non-independent Boltzmann processes will become independent in some sense as $N\to\infty$. While in \cite{MR4419606},   Heydecker,  using the Tanaka's trick \cite{MR512334}, coupled the Kac's particle system characterized by a Poisson driven SDE with the generator $\cL_N$ defined in \eqref{kacg}, named Kac's processes,  with a cutoff Kac's processes to conclude the result, relying also on some stabilization result.  The totally different coupling is probably one of the main reasons that why we have different rate of chaos convergence.

\vip 
Let's also mention some other works \cite{MR3621429,MR3572320,MR3456347,MR3784497} in terms of  propagation of chaos. In \cite{MR3621429}, Fournier and Guillin used the same second coupling strategy to get the rate of chaos for the Landau equation in the hard potentials and Maxwell molecules cases, and in \cite{MR3572320}, Fournier and Hauray considered the singular soft potentials for the Landau equation. In \cite{MR3456347}, Fournier and Mischler gave an optimal rate of $N^{-1/2}$ for {\it Nanbu} particle system, which means that only one particle changes  while  the other keeps at each collision (less meaningful from the physical point of view),  for the Boltzmann equation in the case of hard potentials, hard spheres and Maxwell molecules. Thus, they  can couple the system with a system of i.i.d. Boltzmann processes.  Following their road,  \cite{MR3784497} extended the chaos  result to the singular {\it soft potentials}, even though the rate is far from sharp. Also \cite{Salem19} obtained the convergence without rate for the {\it soft potentials}.

\vip

Compared to the Maxwell molecules case, the main difficulty arises from the term $|v-v_*|^\gamma$ in the cross section for the hard potential,  which complicates the main computation in Section 2. The second difficulty arises from the lack of continuity of the parametrization  of the collision angles,  so that  we need to  handle the cutoff Boltzmann processes when we consider the second coupling, which makes  the computations  in  \cite{MR3456347}  doesn't fit well our case since the particles in the system have only uniformly  finite moments with order greater than $4$! The third difficulty is that the collision in the system is binary which leads to the loss of  independence of the system, this obviously increases the complexity of the problem.
To summarize:  We obtain a  better rate of convergence for hard potentials and hard spheres  even though  this rate is not sharp and uniformly in time. 

\vip

\subsection {plan of the paper.} In Section 2, we give the core estimation on the collision integral. Section 3 focus  on the coupling of the system. In section 4, we proved  the convergence  of  the particle system.
\vip
In the sequel, $C$ stands for a positive constant whose value may change from line to line. When necessary, we will indicate in subscript the parameters it depends on.

\setcounter{equation}{0}

\section{Preliminaries}\label{sec:MainComput}
In this section, we will do some preparations. At first, we rewrite the collision operator by using a substitution $\theta=G(z/|v-v_*|^\gamma)$ to make disappear the velocity-dependence $|v-v_*|^{\gamma}$ in the \emph{rate} like  that in \cite{MR2398952,MR1885616, MR3456347}, which avoids the complexity of formulas in the whole paper.

 \begin{lem}\label{rewriteA} 
Assume \eqref{cs} and \eqref{c2hs} or \eqref{c2}. Recall $G$ defined in \eqref{defH} and $a$ defined in \eqref{dfvprime}.  For $z\in (0,\infty)$, $\varphi\in [0,2\pi)$, $v,v_*\in \rd$, we set
\begin{equation}\label{dfc}
c(v,v_*,z,\varphi):=a[v,v_*,G(z/|v-v_*|^\gamma),\varphi].
\end{equation} 
For any bounded Lipschitz $\phi:\rd\mapsto\rr$,  any $v,v_*\in \rd$
\begin{eqnarray}\label{agood}
\cA\phi(v,v_*)&=&\int_0^\infty dz \int_0^{2\pi}d\varphi
\Big(\phi[v+c(v,v_*,z,\varphi)] -\phi[v]\Big).
\end{eqnarray}
For any $N\geq 1$, $\bv=(v_1,\dots,v_N)\in(\rd)^N$, any bounded measurable
$\phi:(\rd)^N\mapsto\rr$,
\begin{equation}\label{lKgood}
\cL_{N} \phi(\bv)= \frac 1 {2(N-1)} \sum_{i \ne j} \int_0^\infty dz\int_0^{2\pi}d\varphi 
\Big[\phi\big(\bv + c(v_i,v_j,z,\varphi)\be_i+c(v_j,v_i,z,\varphi)\be_j\big) - \phi(\bv)\Big].
\end{equation}
\end{lem}
This Lemma is exactly \cite[Lemma 2.2]{MR3456347} just for different collision operator. We thus  omit the proof here. 
Next, we give the following important estimations which are part of the core computations.

\begin{lem}\label{fundest}
Note that $G$ was defined in \eqref{defH} and $c$ was defined in  \eqref{dfc}. For any $v,v_*,\tv,\tv_* \in \rd$, any $K\in [1,\infty)$, we set $c_K(v,v_*,z,\varphi):=c(v,v_*,z,\varphi)\indiq_{\{z \leq K\}}$.
\vip
\benu[label=\emph{(\roman*)}]
\item  Denote $\Phi_K(x)=\pi \int_0^K (1-\cos G(z/x^\gamma))dz$, $\Psi_K=\pi \int_K^\infty (1-\cos G(z/x^\gamma))dz$. Then
\begin{align*}
&\int_0^\infty \int_0^{2\pi} \Big( 
\big|v+c(v,v_*,z,\varphi)-\tv-c_K(\tv,\tv_*,z,\varphi+\varphi_0(v-v_*,\tv-\tv_*)) \big|^2 - |v-\tv|^2
\Big) d\varphi   dz\\
\leq& I_1^K(v,v_*,\tv,\tv_*)+I_2^K(v,v_*,\tv,\tv_*)+I_3^K(v,v_*,\tv,\tv_*),
\end{align*}
where 
\begin{align*}
I_1^K(v,v_*,\tv,\tv_*):= &2  |v-v_*||\tv-\tv_*| 
\int_0^K \big[G(z/|v-v_*|^\gamma)-G(z/|\tv-\tv_*|^\gamma)\big]^2 dz,\\
I_2^K(v,v_*,\tv,\tv_*):= &-  \big[(v-\tv)+(v_*-\tv_*)\big]\cdot\big[(v-v_*)
\Phi_K(|v-v_*|)-(\tv-\tv_*)\Phi_K(|\tv-\tv_*|)\big],\\
I_3^K(v,v_*,\tv,\tv_*):= &(|v-v_*|^2+2|v-\tv||v-v_*|)\Psi_K(|v-v_*|).
\end{align*}

\item 
\begin{align*}
\int_0^\infty \int_0^{2\pi} \Big( 
\big|v+c_K(v,v_*,z,\varphi)-\tv \big|^2 - |v-\tv|^2
\Big) d\varphi   dz
\leq C(1+|v|^{2+2\gamma}+|\tv|^{2+2\gamma}+|v_*|^{2+2\gamma}). 
\end{align*}
\eenu
\end{lem}

\begin{proof} 
(i) This point largely follows the same arguments as Lemma 3.1 of \cite{MR3456347}.  Roughly speaking, it is mainly derived from the properties of $c$ and $G$ and from the well known trick initiated by Tanaka in  \cite{MR512334} dealing with the Maxwell molecules for the Boltzmann equation. For convenience, we will briefly write down the proof again here. For simplicity, we write $x=|v-v_*|$, $\tx=|\tv-\tv_*|$, $\bar\varphi=\varphi+\varphi_0(v-v_*,\tv-\tv_*)$,
$c=c(v,v_*,z,\varphi)$, $\tc=c(\tv,\tv_*,z,\bar\varphi)$ and
$\tc_K=c_K(\tv,\tv_*,z,\bar\varphi)=\tc\indiq_{\{z\leq K\}}$. Recall the definition of $c(v,v_*,z,\varphi)$ in 
\eqref{dfc}, and the formula \eqref{num}, we have $
\int_0^K \int_0^{2\pi} |c|^2 d\varphi dz =x^2  \Phi_K(x)$ and $\int_0^K \int_0^{2\pi} |\tc|^2 d\varphi dz = \tx^2  \Phi_K(\tx)$. Also recall \eqref{dfvprime} and \eqref{dfc}, using that $\int_0^{2\pi} \Gamma(v-v_*,\varphi) d\varphi=0$, we know that $\int_0^K \int_0^{2\pi} c d\varphi dz= - (v-v_*) \Phi_K(x).$ And in  the same way, 
$$\int_0^K \int_0^{2\pi} \tc d\varphi dz= - (\tv-\tv_*) \Phi_K(\tx), \int_K^\infty \int_0^{2\pi} c d\varphi dz= - (\tv-\tv_*) \Psi_K(\tx).$$ Operating  as Lemma 3.1 of \cite{MR3456347}, we see that 
\beq\label{cesti}
\int_0^K \int_0^{2\pi}(|v+c-\tv-\tc|^2-|v-\tv|^2)d\varphi dz\le I_1^K(v,v_*,\tv,\tv_*)+I_2^K(v,v_*,\tv,\tv_*).
\eeq
And $\int_K^\infty \int_0^{2\pi}(|v+c-\tv|^2-|v-\tv|^2)d\varphi dz\le I_3^K(v,v_*,\tv,\tv_*)$ with $I_3^K(v,v_*,\tv,\tv_*):=(x^2+2|v-\tv|x)\Psi_K(x).$ Which gives Point (i).

 \vip
 (ii)  A straightforward computation gives 
\begin{align*}
    \int_0^\infty \int_0^{2\pi} \Big( 
\big|v+c_K-\tv \big|^2 - |v-\tv|^2
\Big) d\varphi   dz
=\int_0^K \int_0^{2\pi} \Big( 
\big|c|^2 +2 (v-\tv)\cdot c
\Big) d\varphi   dz.
\end{align*}
We know $\int_0^K \int_0^{2\pi} |c|^2 d\varphi dz =x^2  \Phi_K(x)$ and  $\int_0^K \int_0^{2\pi} c d\varphi dz= - (v-v_*) \Phi_K(x)$ according to the computation in (i). Using that $\int_0^\infty G^2(z/x^\gamma)dz=x^\gamma\int_0^\infty G^2(z)dz$, we have
$$
\int_0^K \int_0^{2\pi}\big|c|^2 d\varphi   dz\le\pi|v-v_*|^2\int_0^\infty (1-\cos G(z/|v-v_*|^\gamma))dz\le C|v-v_*|^{2+\gamma},
$$
and
\begin{align*}
\Big|\int_0^K \int_0^{2\pi} c d\varphi dz\Big|\le \pi |v-v_*| \int_0^\infty (1-\cos G(z/|v-v_*|^\gamma))dz 
\le C|v-v_*|^{1+\gamma}.
\end{align*}
This implies Point (ii).
\end{proof}

We now give our main computations. Due to our particle system has a uniform (only depending on the moments of the initial data)  bounded  moments of order  greater than 4,  which enables us to  relax the strict requirements on the moment of the particle system,  but also makes our computation more succinct because we now can allow our estimates to involve powers of  $p$ greater than $4$ of $v,v_*$ and $\tv,\tv_*$. Let's  first deal with the complicated hard potential.

\begin{lem}\label{further}  Assume \eqref{cs} and \eqref{c2} for $\gamma\in(0,1)$ and adopt the notation of Lemma \ref{fundest}.
\vip
(i) For all $q>0$, there is $C_q>0$ 
such that for all $M\geq 1$, all $K\in [1,\infty)$, all $v,v_*,\tv,\tv_* \in \rd$,
\begin{align*}
I_1^K(v,v_*,\tv,\tv_*) \leq & M(|v-\tv|^2+ |v_*-\tv_*|^2) + C_q 
e^{-M^{q/\gamma}} e^{C_q (|\tv|^q+|\tv_*|^q)}.
\end{align*}

\vip
(ii) There is $C>0$ such that for all $K\in [1,\infty)$, all $v,v_*,\tv,\tv_* \in \rd$ and all $z_*\in\rd$,
\begin{align*}
I_2^K(v,v_*,\tv,\tv_*)- I_2^K(v,v_*,\tv,z_*) \leq & C\Big[|v-\tv|^2 + |v_*-\tv_*|^2 \\
& \hskip1cm + |\tv_*-z_*|^2(1+|v|+|v_*|+|\tv|+|\tv_*|+|z_*|)^{2} \Big].
\end{align*}

(iii) For  all $p\ge 2+2\gamma/\nu$, there is $C_{p}>0$ such that for all $K\in [1,\infty)$, all $v,v_*,\tv,\tv_* \in \rd$,
\begin{align*}
I_3^K(v,v_*,\tv,\tv_*) \leq &  C_{p}(1+|v|^{p}+|v_*|^{p}+|\tv|^{p})K^{1-2/\nu}.
\end{align*}
\end{lem}

\begin{proof} (i) For $0<\gamma<1$,  this point exactly follows from Lemma 3.3-(i)  of \cite{MR3456347}. 

\vip
(ii) To bound $I_2^K$, the following proof is the variant of   Lemma  3.3 of \cite{MR3456347}. According to Lemma  3.3 of \cite{MR3456347}, we see that there is $C>0$ for all  $K\ge1$, 
\[\Phi_K(x)\le Cx^\gamma, \quad |\Phi_K(x)-\Phi_K(y)|\le C|x^\gamma-y^\gamma|.\]
Hence, for all $X,Y\in\rd$
\begin{align*}
|X\Phi_K(|X|)-Y\Phi_K(|Y|)|\le C|X-Y||X|^\gamma+C|Y|||X|^\gamma-|Y|^\gamma|.
\end{align*}
Using again that $|x^\gamma-y^\gamma|\le 2|x-y|/(x^{1-\gamma}+y^{1-\gamma})$ for $x>0, y>0$, then 
\beq\label{fi}
|X\Phi_K(|X|)-Y\Phi_K(|Y|)|\le C|X-Y|(|X|^\gamma+|Y|^\gamma).
\eeq
Since $I^K_2$ is anti-symmetric for  $(v,v_*)$ and  $(\tv,\tv_*)$, we have 
\begin{align}\label{topcool1}
\Delta_2^K:=&I_2^K(v,v_*,\tv,\tv_*)-I_2^K(v,v_*,\tv,z_*) \nonumber \\
=& -  \big[(v-\tv)+(v_*-\tv_*)\big]\cdot\big[(\tv-z_*) \Phi_K(|\tv-z_*|)-(\tv-\tv_*)\Phi_K(|\tv-\tv_*|)\big]\\
&+ (\tv_*-z_* )\cdot\big[(v-v_*) \Phi_K(|v-v_*|)-(\tv-z_*)\Phi_K(|\tv-z_*|)\big].\nonumber
\end{align}
By \eqref{fi}, we deduce that
\begin{align*}
\Delta_2^K\leq &C (|v-\tv|+|v_*-\tv_*|) |\tv_*-z_*| (|\tv-z_*|^\gamma+|\tv-\tv_*|^\gamma)\\
&+C |\tv_*-z_* | (|v-\tv|+|v_*-z_*|)(|v-v_*|^\gamma + |\tv-z_*|^\gamma)\\
\leq & C  [(|v-\tv|+|v_*-\tv_*|)^2 +  |\tv_*-z_*|^2 (|\tv-z_*|^\gamma+|\tv-\tv_*|^\gamma)^2 ]\\
&+ C|\tv_*-z_*| (|v-\tv|+|v_*-\tv_*|+|\tv_*-z_*|)(|v-v_*|^\gamma + |\tv-z_*|^\gamma).
\end{align*}
It's clear that the first term is  bounded by
$C  (|v-\tv|^2+|v_*-\tv_*|^2 + |\tv_*-z_*|^2 (1+|\tv| +|\tv_*|+|z_*|)^{2\gamma})$ which fits the statement, since
$2\gamma< 2$. Next, we bound the second term by
\begin{align*}
&C|\tv_*-z_* |^2 (|\tv-z_*|+ |v-v_*|)^\gamma\\
&+ C|\tv_*-z_* |(|v-\tv|+|v_*-\tv_*|) (|\tv-z_* |+ |v-v_*|)^\gamma.
\end{align*}
Using that $xyz^\gamma\leq (xz^\gamma)^2
+y^2$ (for the last line), we obtain 
\begin{align*}
&C|\tv_*-z_* |^2 (|\tv-z_* |+ |v-v_*|)^{\gamma}\\
&+ C (|v-\tv|+|v_*-\tv_*|)^2 + |\tv_*-z_* |^2  (|\tv-z_* |+ |v-v_*|)^{2\gamma},
\end{align*}
which is bounded by 
\begin{align*} 
C (|v-\tv|^2+|v_*-\tv_*|^2) + C |\tv_*-z_* |^2 (1+|\tv|+|z_*| + |v| + |v_*|)^2.
\end{align*}
One easily concludes the point using that $2>2\gamma$.

\vip

We finally check point (iii). Recalling  $\Psi_K=\pi \int_K^\infty (1-\cos G(z/x^\gamma))dz$ and using \eqref{eG}, we have $1-\cos(G(z/x^\gamma))\leq G^2(z/x^\gamma)
\leq C (z/x^\gamma)^{-2/\nu}$, whence $\Psi_K(x) \leq C x^{2\gamma/\nu} \int_K^\infty z^{-2/\nu}dz
= Cx^{2\gamma/\nu} K^{1-2/\nu}$. Thus
\begin{align}\label{ar2}
I_3^K(v,v_*,\tv,\tv_*) \leq & C(|v-v_*|^2+|v-\tv|^2)|v-v_*|^{2\gamma/\nu} K^{1-2/\nu}\\
\le &C(|v-v_*|^{2+2\gamma/\nu}+|v-\tv|^2|v-v_*|^{2\gamma/\nu})K^{1-2/\nu},\nonumber
\end{align}
using the Young inequality that $|v-\tv|^2|v-v_*|^{2\gamma/\nu}\leq |v-\tv|^{2+2\gamma/\nu}+
|v-v_*|^{2+2\gamma/\nu}$, we observe for  any $p\ge \max{(4,  2+2\gamma/\nu)}$, 
\begin{align*}
I_3^K(v,v_*,\tv,\tv_*)&\le C(|v-v_*|^{2+2\gamma/\nu}+|v-\tv|^{2+2\gamma/\nu})K^{1-2/\nu}\\
&\le C_{p}(1+|v|^{p}+|v_*|^{p}+|\tv|^{p})K^{1-2/\nu}.
\end{align*}
This concludes point (iii).

\end{proof}

We now consider the hard shpere case $\gamma=1$ under the assumption  \eqref{cs} 
and \eqref{c2hs}.   Specifically,  the computation in the following Lemma is carried out with  $G(z)=\max{(\pi/2-z,0)}$.

\begin{lem}\label{furtherhs} Assume \eqref{cs} 
and \eqref{c2hs} for $\gamma=1$ and adopt the notation of Lemma \ref{fundest}. Then

\vip
(i) For all $q>0$, there is $C_q>0$ 
such that for all $M\geq 1$, all $K\in [1,\infty)$, all $v,v_*,\tv,\tv_* \in \rd$,
\begin{align*}
I_1^K(v,v_*,\tv,\tv_*) \leq & M(|v-\tv|^2+ |v_*-\tv_*|^2) + C_q K(|v|+|v_*| )
e^{-M^{q}} e^{C_q (|\tv|^q+|\tv_*|^q)}.
\end{align*}

\vip

(ii) For all $q>0$, there is $C_q>0$ 
such that for all $M\geq 1$, all $K\in [1,\infty)$, all $v,v_*,\tv,\tv_* \in \rd$,
\begin{align*}
I_2^K(v,v_*,\tv,\tv_*)- I_2^K(v,v_*,\tv,z_*) \leq & M(|v-\tv|^2 + |v_*-\tv_*|^2)
+  C |\tv_*-z_*|^2(1+|\tilde v|+|\tv_*|+|z_*|)^{2}\\
&+ C_q (1+|v|+|v_*|) Ke^{- M^q } e^{ C_q(|\tv|^q+|\tv_*|^q+|z_*|^q)}.
\end{align*}

\vip

(iii) For all $p>0$, there is $C_{p}>0$ such that for all $K\in [1,\infty)$, all $v, v_*,\tv,\tv_* \in\rd$,
\begin{align*}
I_3^K(v,v_*,\tv,\tv_*) \le  C_{p} K^{-p}(1+\abs{v}^{p+3}+ \abs{\tv}^{p+3}+\abs{v_*}^{p+3} ).
\end{align*}

\end{lem}

\begin{proof}
(i) and (ii) exactly follow from \cite[Lemma 3.4]{MR3456347} due to $I^K_1$ is symmetric and  $I^K_2$ is anti-symmetric for  $(v,v_*)$ and  $(\tv,\tv_*)$. For (iii),  recalling that $\Psi_K=\pi \int_K^\infty (1-\cos G(z/x^\gamma))dz$ and that $G(z)=\max{(\pi/2-z,0)}$  supported on $(0,\pi/2)$. Hence $\Psi_K(x)\le \pi\int_K^\infty G^2(z/x)dz=\pi\int_K^{\pi/2}(\pi/2-z/x)^2\indiq_{\{x\ge 2K/\pi\}}dz \le5x \indiq_{\{x\ge K/2\}}.$ Thus for any positive $p$, $\Psi_K(x)\le  2^{p+3} x^{p+1}K^{-p}. $ Whence 
    \begin{align*}
        I_3^K(v,v_*,\tv,\tv_*) \leq& 2^{p+3} K^{-p}\abs{v-v_*}^{p+1}
        (\abs{v -v_*}^2 + 2\abs{v - \tv}\abs{v -v_*}) \\
        \leq& C_{p} K^{-p}(1+\abs{v}^{p+3}+ \abs{\tv}^{p+3}+\abs{v_*}^{p+3} ).
    \end{align*}

\end{proof}

Next, we will  bound $I_1^K, I_2^K$ with $v_*=\tv_*$ for $\gamma\in(0,1]$ which is crucial for getting a decoupling estimate for the system of non-independent cutoff Boltzmann processes.
\begin{lem}\label{IK247}
Assume \eqref{cs} and \eqref{c2hs} or \eqref{c2} for $0<\gamma\le1$ and recall $I_i^K$’s defined in Lemma \ref{fundest}. For any $q>0$, there is $C_q>0$ such that for  all $M\ge1$, $K\ge1$, and all $v,v_*,\tv \in \rd$, we have 
$$I_1^K(v,v_*,\tv,v_*)+I_2^K(v,v_*,\tv,v_*) \le CM|v-\tv|^2+C_qe^{-M^{q/\gamma}}e^{C_q(|v|^q+|\tv|^q+|v_*|^q)}.$$

\end{lem}

\begin{proof}
For $0<\gamma<1$, taking $v_*=\tv_*$ in Lemma \ref{further}-(i) implies directly that $$I_1^K(v,v_*,\tv,v_*)\le M|v-\tv|^2+C_qe^{-M^{q/\gamma}}e^{C_q(|\tv|^q+|v_*|^q)}.$$
For $\gamma=1$, \eqref{c3} implies
\begin{align*}
I_1^K(v,v_*,\tv,v_*)\leq& 2 c_4 |v-v_*||\tv-v_*| \frac{(|v-v_*|-|\tv-v_*|)^2}
{|v-v_*|+|\tv-v_*|} \\
\leq & 2c_4 (|v-v_*| \land |\tv-v_*|)
(|v-v_*|-|\tv-v_*|)^2\\
\leq &2c_4 (|v-v_*| \land |\tv-v_*|)|v-\tv|^2.
\end{align*}
Then for any $M\ge1$ and fixed $q>0$,
\begin{align*}
I_1^K(v,v_*,\tv,v_*)\leq &M|v-\tv|^2+2c_4|v-\tv|^2(|v-v_*| \land |\tv-v_*|)\indiq_{\{2c_4(|v-v_*| \land |\tv-v_*|)\ge M\}}\\
\le & M|v-\tv|^2+C|v-\tv|^2|v-v_*|e^{-M^q}e^{|v-v_*|^q}\\
\le &M|v-\tv|^2+C_qe^{-M^q}e^{2|v-v_*|^q+|v-\tv|^q}\\
\le& M|v-\tv|^2+C_qe^{-M^q}e^{C_q(|v|^q+|\tv|^q+|v_*|^q)}.
\end{align*}
Using \eqref{fi} with $v_*=\tv_*$, we also find 
$$I_2^K(v,v_*,\tv,v_*)\le C|v-\tv|^2(|v-v_*|^\gamma+|\tv-v_*|^\gamma).$$
Similarly, for $M\ge1$, and  any fixed $q>0$,
\begin{align*}
I_2^K(v,v_*,\tv,v_*)&\le M|v-\tv|^2+C|v-\tv|^2(|v-v_*|^\gamma+|\tv-v_*|^\gamma) \indiq_{\{C(|v-v_*|^\gamma+|\tv-v_*|^\gamma)\ge M\}}\\
&\le M|v-\tv|^2+C_q e^{3C_q(|v|^q+|v_*|^q+|\tv|^q)}e^{-M^{q/\gamma}}.
\end{align*}
Thus, we get 
$$I_2^K(v,v_*,\tv,v_*)\le CM|v-\tv|^2+C_qe^{-M^{q/\gamma}}e^{C_q(|v|^q+|v_*|^q+|\tv|^q)},$$
which concludes the proof.
\end{proof}

At the end of the section, let's look at a version  of  Povzner inequality  \cite{MR0142362}, see for instance \cite{MR1697562,MR1450762,MR3476632} and the references therein for other versions. This inequality enables us to derive the uniform bound for the moments of Kac's particle system.

\begin{lem} \label{Povzner}
For any $p>2$ and any $v,v_*\in \rd$, we have 
 \begin{align*}
 &\int_0^\infty \int_0^{\pi/2}(|v'|^p+|v'_*|^p-|v|^p-|v_*|^p)d\varphi dz \\
 &\le -A_p|v-v_*|^\gamma(|v|^p+|v_*|^p)+\tilde{A}_p|v-v_*|^\gamma(|v|^{p-2}|v_*|^2+|v_*|^{p-2}|v|^2),
 \end{align*}
 where $A_p=\int_0^{\pi/2}[1-(\cos(\theta/2))^p-(\sin(\theta/2))^{p}]\beta(\theta)d\theta>0$, and $\tilde{A}_p>0$ is a constant depending only on $p$.
\end{lem}
 \begin{proof}
 We will give a sketch of the proof. For any fixed $v,v_*$, $x=|v-v_*|$, $c$ defined in \eqref{dfc}, $v',v_*'$ defined in \eqref{dfvprime} with $c$, then from a straightforward calculation, 
 \beq\label{2power}
 |v'|^2=\frac{1+\cos G(z/x^\gamma)}{2}|v|^2+\frac{1-\cos G(z/x^\gamma)}{2}|v_*|^2+\sin G(z/x^\gamma)v\cdot\Gamma(v-v_*,\varphi).
 \eeq
 Since $|\Gamma(v-v_*,\varphi)|=|v-v_*|\le |v|+|v_*|$ and the fact that $v\cdot\Gamma(v-v_*,\varphi)=v_*\cdot\Gamma(v-v_*,\varphi)$(due to $\Gamma(v-v_*,\varphi)$ is orthogonal to $v-v_*$), so 
 \[|v\cdot\Gamma(v-v_*,\varphi)|\le(|v|\wedge|v_*|)|\Gamma(v-v_*,\varphi)|\le2|v||v_*|.\]
 Then for any $p>2$, we take the $p/2$-th power to \eqref{2power} and find 
 \[|v'|^p \le \Big(\cos \frac{G\big(z/x^\gamma)}{2}\Big)^p|v|^p+\Big(\sin \frac{G\big(z/x^\gamma)}{2}\Big)^p|v_*|^p+C_p\sin G(z/x^\gamma)(|v|^{p-2}|v_*|^2+|v_*|^{p-2}|v|^2).\]
 With the same argument for $v_*’$ by exchanging the roles of $v$ and $v_*$, we obtain 
 \begin{align*}
 |v'|^p+|v'_*|^p-|v|^p-|v_*|^p \le &-\kappa(p,G(z/x^\gamma))(|v|^p+|v_*|^p)\\
 &+C_p\sin G(z/x^\gamma)(|v|^{p-2}|v_*|^2+|v_*|^{p-2}|v|^2),
 \end{align*}
 where $\kappa(p,\theta)=\Big(1-\big(\cos{\frac{\theta}{2}}\big)^p-\big(\sin\frac{\theta}{2}\big)^{p}\Big)>0$ and $\theta=G(z/x^\gamma)$. We  thus conclude the proof.
  
 \end{proof}

\section{The first coupling}
\setcounter{equation}{0} 
\label{sec:ConvCutoff}

In this section, we construct a suitable first coupling between the Kac's particle system and the solution to \eqref{be} inspired by \cite{MR3476628,MR3769742,MR3456347}.  Roughly speaking, we  need to explicitly build the processes related to the particle system and the solution of  \eqref{be} respectively which satisfy stochastic differential equation driven by the same Poisson random measure.

\subsection{The Boltzmann process}
Following Tanaka \cite{MR536022,MR512334}  for Maxwell molecules,  we will consider  the following  probabilistic realization of the weak solution of \eqref{be} given by Fournier \cite{MR3313757} for $\gamma\in(0,1]$,  see also \cite[Proposition 4.1]{MR3456347}.
\begin{prop}\label{Bp}
For $\gamma\in(0,1]$, suppose \eqref{cs}, \eqref{c4}, \eqref{c2hs} or \eqref{c2}. Let $f_0\in\cP_2(\rd)$ and satisfy \eqref{c5}. Let $(f_t)_{t\ge0}$ be the corresponding unique solution of \eqref{be}. Then there exists, on some probability space, a $f_0$-distributed random variable $W_0$, an independent Poisson measure $M(ds,dv,dz,d\varphi)$ with intensity $ds f_s(dv) dzd\varphi$, and a unique (c\`{a}dl\`{a}g adapted) process $(W_t)_{t\ge0}$ satisfying 
\beq\label{tbp}
W_t=W_0+\intot\int_{\rd}\int_0^\infty\int_0^{2\pi} c(W_\sm,v, z,\varphi)M(ds,dv,dz,d\varphi).
\eeq
Moreover, $W_t$ is $f_t$-distributed for each $t\ge0$.
\end{prop}
The process $(W_t)_{t\ge0}$ is called the {\it Boltzmann process}, which describes the velocity of a typical particle in the dilute gas. The process jumps when the particle collides with the others.  This Proposition is directly implied by Proposition 5.1 of \cite{MR3313757},  but with different formulation of SDE of the process which is equivalent to that determined by \eqref{tbp} in law, for the same Boltzmann equation under fewer assumption on the initial data $f_0$.
\vip

Unfortunately, the weak existence of solution to the equation \eqref{tbp} cannot enable us to perform the  coupling procedure, especially for the second coupling, we thus need to work with a cutoff Boltzmann process instead. To do so, let's first  introduce a similar result to Theorem \ref{wp} for a cutoff Boltzmann equation below. 
\begin{rem}\label{cutws}
Given $K\ge1$, consider the Boltzmann equation \eqref{be} with collision kernel  $B(|v-v_*|,\theta)$  replaced by $B(|v-v_*|,\theta)\indiq_{\{\theta\ge G(K/|v-v^*|^\gamma)\}}$, where $G$  is defined in \eqref{defH}.  Let $f_0\in \cP_2(\rd)$  and satisfying the assumption \eqref{c5}. Consider the unique solution $(f_t)_{t\ge0}$ to \eqref{be} in Theorem \ref{wp}, then the cutoff Boltzmann equation has a  unique solution  $(f_t^K)_{t\ge0}\in C([0,\infty),\cP_2(\rd))$ starting from $f_0$,  for which Theorem \ref{wp} holds as well. Moreover, for all $t\ge0$,
\beq\label{cutf}
\cW_2^2(f_t^K,f_t)\le CtK^{1-2/\nu}.
\eeq
\end{rem}
This remark can be implied directly by  \cite[Theorem 1.5]{MR2871802} since  our collision kernel is angular  cutoff, i.e. $\int_{\Sp^{2}}B(|v-v_*|,\theta)\indiq_{\{\theta\ge G(K/|v-v^*|^\gamma)\}} d\sigma=K<\infty$. \eqref{cutf} can be shown by a similar way of Lemma 15 in \cite{MR3769742}, see \cite[Theorem  1]{MR4419606}, also \cite[Lemma 4.2]{MR2398952} for $\gamma\in(-3,0]$ . Thus we have $\lim\limits_{K\to\infty}\cW_2^2(f_t^K,f_t)=0$.

\begin{rem}\label{cutbp}
Given $K\ge1$,  recalling $c_K(v,v_*, z,\varphi):=c(v,v_*,z,\varphi)\indiq_{\{z \leq K\}}$.
Consider $(f_t^K)_{t\ge0}$  stated above, $M^K(ds,dv,dz,d\varphi)$ is a Poisson measure  with  intensity $ds f_s^K(dv) dzd\varphi$. Then the following cutoff nonlinear SDE 
\beq\label{cbp}
W_t^K=W_0+\intot\int_{\rd}\int_0^\infty\int_0^{2\pi} c_K(W_\sm^K,v, z,\varphi)M^K(ds,dv,dz,d\varphi),
\eeq
has a unique (in law)  strong solution $(W_t^K)_{t\ge0}$. Moreover, $W_t^K$ is $f_t^K$-distributed for each $t\ge0$.  
\end{rem}
The  well-posedness for \eqref{cbp} is obvious, because the Poisson measure involved in \eqref{cbp} is actually finite (because $c_K=c\indiq_{\{z \leq K\}}$), so that the equation is  nothing but a recursive equation. Moreover,  a standard procedure enables  us to know that  any solution to \eqref{cbp}  is $f_t^K$-distributed.

\subsection{A SDE for the Kac's particle system}
We give a  stochastic differential equation for Kac's particle system with generator \eqref{kacg}.

\begin{prop}\label{kpst}
Assume \eqref{cs}, \eqref{c4}, \eqref{c2hs} or \eqref{c2}. Let $f_0 \in \cP_2(\rd)$, $N\geq 1$.
Consider a family $(V^i_0)_{i=1,\dots N}$ of i.i.d. $f_0$-distributed random variables and an independent
family  $(O_{ij}(ds,dz,d\varphi))_{1\le i<j\le N}$ of 
Poisson measures 
with intensity 
$\frac{1}{(N-1)}ds dz d\varphi$. For $1\le j<i\le N$, we put $O_{ij}(ds,dz,d\varphi)=O_{ji}(ds,dz,d\varphi).$  And  set $O_{ii}(ds,dz,d\varphi)=0$ for $i=1,\dots, N.$
There exists a unique (c\`adl\`ag and adapted)  solution to
\begin{align}\label{pssde}
V^{i}_t=V^i_0 + \sum_{j=1}^N\intot\int_0^\infty\int_0^{2\pi} c(V^{i}_\sm,V^{j}_\sm,z,\varphi) 
O_{ij}(ds,dz,d\varphi),
\quad i=1,\dots,N.
\end{align}
Furthermore, $(V^{i}_t)_{i=1,\dots,N, t\geq 0}$ is $(\rd)^N$-valued Markov process with generator  $\cL_{N}$. And the system is almost surely conservative:  for all $t\ge 0$, it holds that $\sum_{i=1}^N V^{i}_t=\sum_{i=1}^N V^{i}_0 $ and $\sum_{i=1}^N|V^{i}_t|^2=\sum_{i=1}^N|V^{i}_0|^2 .$ 
\end{prop}

\begin{proof}
Fix $N\ge2$,  $f_0 \in \cP_2(\rd)$ and  a family $(V^i_0)_{i=1,\dots N}$ of i.i.d. $f_0$-distributed random variables. 

\vskip2mm
\emph{Step} 1. We first prove the conservation of  the system. Assume that $(V^{i}_t)_{i=1,\dots,N, t\geq 0}$ is a Markov process with generator $\cL_{N}$ starting from $(V^i_0)_{i=1,\dots N}$ solves \eqref{pssde} with some Poisson measures introduced above.  By (\ref{pssde}), we have 
\begin{align*}
    \sum_{i=1}^N V^{i}_t=\sum_{i=1}^N V^{i}_0+\sum_{i=1}^N\sum_{j=1}^N\intot\int_0^\infty\int_0^{2\pi} c(V^{i}_\sm,V^{j}_\sm,z,\varphi) O_{ij}(ds,dz,d\varphi).
\end{align*}
Since $ O_{ij}(ds,dz,d\varphi)= O_{ji}(ds,dz,d\varphi)$ and  $c(V^{i}_\sm,V^{j}_\sm,z,\varphi)=-c(V^{j}_\sm,V^{i}_\sm,z,\varphi)$. It's  obvious that \begin{align*}
    \sum_{i=1}^N\sum_{j=1}^N\intot\int_0^\infty\int_0^{2\pi} c(V^{i}_\sm,V^{j}_\sm,z,\varphi) 
O_{ij}(ds,dz,d\varphi)=0.
\end{align*}
We thus get $\sum_{i=1}^N V^{i}_t=\sum_{i=1}^N V^{i}_0.$ Next, the It\^o's formula implies that
\begin{align*}
\sum_{i=1}^N|V^{i}_t|^2=&\sum_{i=1}^N|V^i_0|^2 + \sum_{i=1}^N\sum_{j=1}^N \intot\int_0^\infty\int_0^{2\pi} 
\Big(|V^{i}_\sm+c(V^{i}_\sm,V^{j}_\sm,z,\varphi)|^2 
- |V^{i}_\sm|^2\Big) O_{ij}(ds,dz,d\varphi)\\
=&\sum_{i=1}^N|V^i_0|^2 + \sum_{1\le i<j\le N} \intot\int_0^\infty\int_0^{2\pi} 
\Big(|V^{i}_\sm+c(V^{i}_\sm,V^{j}_\sm,z,\varphi)|^2 \\
&\hskip1cm+|V^{j}_\sm-c(V^{i}_\sm,V^{j}_\sm,z,\varphi)|^2- |V^{i}_\sm|^2-|V^{j}_\sm|^2\Big) O_{ij}(ds,dz,d\varphi)
\end{align*}
Recall the definition of $c(v,v^*,z,\varphi)$ and the fact that $|c(v,v^*,z,\varphi)|=\sqrt{\frac{1-\cos\theta}{2}}|v-v^*|$. Since $(v-v^*)$ and $\Gamma(v-v^*,\varphi)$ are  orthogonal, we get   for $i\ne j,$  
\begin{align*}
&|V^{i}_\sm+c(V^{i}_\sm,V^{j}_\sm,z,\varphi)|^2 +|V^{j}_\sm-c(V^{i}_\sm,V^{j}_\sm,z,\varphi)|^2- |V^{i}_\sm|^2-|V^{j}_\sm|^2\\
&=2(V^{i}_\sm-V^{j}_\sm)\cdot c(V^{i}_\sm,V^{j}_\sm,z,\varphi)+2|c(V^{i}_\sm,V^{j}_\sm,z,\varphi)|^2\\
&=-(1-\cos\theta)|V^{i}_\sm-V^{j}_\sm|^2+(1-\cos\theta)|V^{i}_\sm-V^{j}_\sm|^2=0,
\end{align*}
where $\theta= G(z/|V^{i}_\sm-V^{j}_\sm|^\gamma).$
We thus conclude that the system preserves the  energy.
\vskip2mm
\emph{Step} 2. The proof to get  the  weak existence and weak uniqueness of \eqref{pssde} is standard, we will sketch the proof as in Proposition 1.2-(ii) of \cite{MR3456347}. First, we introduce a cutoff case: for $K\ge1$,  consider 
\begin{align}\label{cpssde}
V^{i,K}_t=V^i_0 + \sum_{j=1}^N\intot\int_0^\infty\int_0^{2\pi} c_K(V^{i,K}_\sm,V^{j,K}_\sm,z,\varphi) 
O_{ij}(ds,dz,d\varphi),
\quad i=1,\dots,N.
\end{align}
where $c_K = c\indiq_{\{z\le K\}}$. It's obvious that the strong existence and uniqueness for \eqref{cpssde} hold,
 since the Poisson measures involved in \eqref{cpssde} are finite. Set $\bV^K=(V^{1,K},\cdots,V^{N,K})$ as solution of \eqref{cpssde}. The existence of a solution in law to \eqref{pssde} is  not hard to check. Indeed, consider any subsequence of $\bV^K$, denote  still  by $\bV^K$, we can show that the laws of  $\bV^K$ are tight by using Aldous's  criterion  \cite{MR0474446},   since the second moment estimates are obvious  uniformly in $K$ and  $N$ due  to the conservation proved above. Then using martingale problems methods and classic probability space enlargement arguments, we thus conclude that  the limit of $\bV^K$  determined by \eqref{cpssde} as $K\to\infty$ gives the  weak existence of solution  to  \eqref{pssde}.
 
 \vskip2mm
\emph{Step} 3. To prove  the uniqueness in law for the weak solution to \eqref{pssde}, it suffices to show  that  any weak  solution $\bV=(V^{i}_t)_{i=1,\dots,N, t\geq 0}$ to \eqref{pssde} can be approximated by the strong solution to the cutoff equation \eqref{cpssde} as the cutoff level goes to infinity. Especially, for $K\ge1$, consider the solution $\hat{\bV}^K$ to 
\begin{align*}
\hat{V}^{i,K}_t=V^i_0 + \sum_{j=1}^N\intot\int_0^\infty\int_0^{2\pi} c_K(\hat{V}^{i,K}_\sm,\hat{V}^{j,K}_\sm,z,\varphi+\varphi_{s,i,j}) 
O_{ij}(ds,dz,d\varphi),
\quad i=1,\dots,N,
\end{align*}
 where $\varphi_{s,i,j}:=\varphi_0(V^{i}_\sm-V^{j}_\sm,\hat{V}^{i,K}_\sm-\hat{V}^{j,K}_\sm)$.  The solution to this equation obviously exists and is unique since the Poisson measures  are finite recalling $c_K = c\indiq_{\{z\le K\}}$. Notice that $\hat{\bV}^K$ is a Markov process starting from $(V^i_0)_{i=1,\dots N}$  with the same law as  $\bV^K$. The only difference between $\hat{\bV}^K$ and $\bV^K$ is the term $\varphi_{s,i,j}$, while  this does not change the law of the system, see Lemma 4.4-(ii) \cite{MR3456347}. So the law  of $\hat{\bV}^K$  is unique. Next, due to Lemma 5.1 of \cite{MR3456347}, using It\^o's formula and Gr\"onwall's lemma, we are then able to conclude that, by a similar computation to Step 2 of proof of Proposition 1.2-(ii) of \cite{MR3456347} which is omitted here, for  each $T>0$,
 \[\lim_{K\to\infty}\sup_{[0,T]}\E[|\bV_t-\hat{\bV}_t^K|^2]=0,\]
which implies that  $\hat{\bV}^K$ converges to  $\bV$ in law as $K\to \infty$ and exactly $\bV$ is the unique Markov process with generator $\cL_N$.
\end{proof}
Next, we will give the following important moment estimation of the particle system. This result is inspired by the \cite[Corollary 17]{MR3769742} for the Maxwell molecules, see also \cite[Lemma 5.3]{MR3069113} with bounded initial energy.
\begin{prop}\label{momentest}
For $\gamma\in(0,1]$, suppose \eqref{cs}, \eqref{c4}, \eqref{c2hs} or \eqref{c2}. Let $p\ge2$. Consider $N\ge2$ and recall $(V_t^{i})_{i=1,\cdots,N,t\ge0}$ introduced in \eqref{ps}. Then there is a constant $C_p>0$ such that
$$\sup_{t\ge0}\mathbb{E}[|V^{1}_t|^p]\le (1+C_{p}) \E[|V^{1}_0|^{p}].$$
\end{prop}
\begin{proof} 
 We  define a random variable $\mathcal{E}_0:=\frac{1}{N}\sum_{i=1}^N|V_0^i|^2$ and set $u_t:=\mathbb{E}[|V^{1}_t|^p|\mathcal{E}_0]$ for each $t\ge0$.  By exchangeability, we have $u_t=\frac{1}{N}\sum_{i=1}^N\mathbb{E}[|V^{i}_t|^p|\mathcal{E}_0].$ By It\^o's formula, we get
\begin{align*}
\sum_{i=1}^N\mathbb{E}[|V^{i}_t|^p]
= \sum_{i=1}^N\mathbb{E}[|V^i_0|^p] &+ \frac {1} {2(N-1)} \sum_{1\le i<j\le N}^N\intot\int_0^\infty\int_0^{2\pi} 
\mathbb{E}\Big[|V^{i}_s+c(V^{i}_s,V^{j}_s,z,\varphi)|^p\\
& +|V^{j}_s+c(V^{j}_s,V^{i}_s,z,\varphi)|^p- |V^{i}_s|^p-|V^{j}_s|^p\Big] dsdzd\varphi.
\end{align*}
Using  Povzner's inequality in Lemma \ref{Povzner} and  exchangeability, we have
\begin{align*}
   \frac{du_t}{dt}
&\le  \frac{1}{2N(N-1)}\sum_{1\le i<j\le N}^N \mathbb{E}\Big[|V^{i}_t-V^{j}_t|^\gamma\\
&\hskip0.3cm\times\Big(-A_p(|V^{i}_t|^p+|V^{j}_t|^p)+\tilde{A}_p(|V^{i}_t|^{p-2}|V^{j}_t|^2+|V^{j}_t|^{p-2}|V^{i}_t|^2)\Big)\Big|\mathcal{E}_0\Big]\\
&=\frac{1}{2}\mathbb{E}\Big[|V^{1}_t-V^{2}_t|^\gamma\\
&\hskip0.5cm\times\Big(-A_p(|V^{1}_t|^p+|V^{2}_t|^p)+\tilde{A}_p(|V^{1}_t|^{p-2}|V^{2}_t|^2+|V^{2}_t|^{p-2}|V^{1}_t|^2\Big)\Big|\mathcal{E}_0\Big].
\end{align*}
However, using that $|v-w|^\gamma\ge \big||v|-|w|\big|^\gamma\ge |v|^\gamma-|w|^\gamma \hbox{and}\,  |v-w|^\gamma\le |v|^\gamma+|w|^\gamma,$ we have 
\begin{align*}
     &\mathbb{E}\Big[|V^{1}_t-V^{2}_t|^\gamma\Big(-A_p(|V^{1}_t|^p+|V^{2}_t|^p)\\
     &\hskip2cm+\tilde{A}_p(|V^{1}_t|^{p-2}|V^{2}_t|^2+|V^{2}_t|^{p-2}|V^{1}_t|^2)\Big)\Big|\mathcal{E}_0\Big]\\
\le& -A_p\mathbb{E}\Big[|V^{1}_t|^{p+\gamma}+|V^{2}_t|^{p+\gamma}\Big|\mathcal{E}_0\Big]+3\tilde{A}_p\mathbb{E}\Big[|V^{1}_t|^{p+\gamma-2}|V^{2}_t|^2+|V^{2}_t|^{p+\gamma-2}|V^{1}_t|^2\Big|\mathcal{E}_0\Big]\\
\le&  -2A_p\mathbb{E}\Big[|V^{1}_t|^{p+\gamma}\Big|\mathcal{E}_0\Big]+6\tilde{A}_p\mathbb{E}\Big[|V^{1}_t|^{p+\gamma-2}|V^{2}_t|^2\Big|\mathcal{E}_0\Big].
\end{align*}
Besides, we also find that 
\begin{align*}
\mathbb{E}\Big[|V^{1}_t|^{p+\gamma-2}|V^{2}_t|^2\Big|\mathcal{E}_0\Big]&=\mathbb{E}\Big[|V^{1}_t|^{p+\gamma-2}\frac{1}{N-1}\sum_{i=2}^N|V^{i}_t|^2\Big|\mathcal{E}_0\Big]\\
&\le 2\mathbb{E}\Big[|V^{1}_t|^{p+\gamma-2}\frac{1}{N}\sum_{i=1}^N|V^{i}_t|^2\Big|\mathcal{E}_0\Big]\le 2\mathcal{E}_0u_t^{1-(2-\gamma)/p},
\end{align*}
where  we  have used the conditional H\"older's inequality,  the fact that $\frac{1}{N} \sum_{i=1}^N|V^{i}_t|^2=\mathcal{E}_0$ and the exchangeability. We thus deduce that 
\begin{align*}
   \frac{du_t}{dt}&\le-A_p\E\Big[|V^{1}_t|^{p+\gamma}\Big|\mathcal{E}_0\Big]+6\mathcal{E}_0u_t^{1-(2-\gamma)/p}\\
&\le-A_p u_t^{(p+\gamma)/p}+6\tilde{A}_p\mathcal{E}_0u_t^{1-(2-\gamma)/p},
   \end{align*}
by applying Jensen's inequality to the last line. Whence  by Lemma 6.3 in \cite{MR1478067}, we deduce that $u_t\le \max\{u_0, x_*\}$, where $x_*=(6\tilde{A}_p\mathcal{E}_0/A_p)^{p/2}$ is the unique positive root determined by  $-A_p u_t^{(p+\gamma)/p}+6\tilde{A}_p\mathcal{E}_0u_t^{1-(2-\gamma)/p}=0$. This implies 
\[u_t\le u_0+C_p\mathcal{E}_0^{p/2},\]
with a constant $C_p>0$ depending only on $A_p$ and $\tilde{A}_p$. Finally, according to the conditional expectation and applying the Jensen's inequality to the term $\frac{1}{N}\sum_{i=1}^N|V_0^i|^2$ , we get
\begin{align*}
\mathbb{E}[|V^{1}_t|^p]=\E(u_t)&\le \mathbb{E}\Big[\E\big(|V^{1}_0|^p|\mathcal{E}_0\big)+C_p\mathcal{E}_0^{p/2}\Big]\\
&\le\E\big(|V^{1}_0|^p\big)+C_p\E\left[\Big(\frac{1}{N}\sum_{i=1}^N|V_0^i|^2\Big)^{p/2}\right]\\
&\le\E\big(|V^{1}_0|^p\big)+C_p\E\left[\frac{1}{N}\sum_{i=1}^N|V_0^i|^p\right],
\end{align*}
which completes the proof by the exchangeability.

\end{proof}

\subsection{The coupling}
To get the convergence of the Kac's particle system to the weak solution $f$ of \eqref{be}, we need to choose some $f$-distributed random variable in \eqref{cbp} in an optimal way such that it is close to the particle in the system colliding with the other in the system coupled with the cutoff Boltzmann process determined by \eqref{cbp}. Following the ideas of Lemma 3 of \cite{MR3476628}, we have the following variant of optimal coupling.

\vip

\begin{lem}\label{ww2p}
Let $f_t\in C([0,\infty),\cP_2(\rd))$ has a density for all $t>0$.  Consider $N\ge 2$, $\bw\in(\rd)_{\bullet}^N$, with $(\rd)_{\bullet}^N:=\{\bw\in(\rd)^N: w_i\ne w_j, \forall i\ne j\}$. For $i,j=1,...,N$, there exists an $\rd$-valued function $\Pi^{ij}_t(\bw,\alpha),$ measurable in $(t,\bw,\alpha)\in[0,\infty]\times (\rd)^N\times[0,1]$  enjoying the following properties:

(i) For any $\bw\in (\rd)_{\bullet}^N,\ t\ge 0,\ i=1,...,N$, $\int_0^1 \frac{1}{N-1}\sum_{j\ne i}^N|\Pi_t^{ij}(\bw,\alpha)-w_j|^2 d\alpha=\cW^2_2(f_t, \bar{\mu}_{\bw}^{i})$
with $\bar{\mu}_{\bw}^{i}=\frac{1}{N-1}\sum_{j\ne i}\delta_{w_j}$.

\vip

(ii) For any exchangeable random vector $\bY\in (\rd)_\bullet^N,\ t\ge 0,\ i=1,...,N$, any $j\ne i$, and for any bounded measurable function $\phi,$  we have $\E[\int_0^1\phi(\Pi^{ij}_t(\bY,\alpha))d\alpha]=\int_\rd\phi(u)f_t(du)$.

\end{lem}

\begin{proof}
For any fixed $(i,j)\in \{1,...,N\}^2$ and $i\ne j$, we are going to construct a measurable mapping $\Pi_t^{ij}:\mathbb{R}_+\times(\rd)_\bullet^N\times[0,1]$ as   $(t,\bw,\alpha)\mapsto\Pi_t^{ij}(\bw,\alpha).$ 
\vip
We fix $n\ge1$. For $\be\in (\rd)_\bullet^n$, we define $\mu_{\be}=\frac{1}{n}\sum_{i=1}^n\delta_{e_i},$ the  empirical measure associated to $\be$. Note that $f_t\in C([0,\infty),\cP_2(\rd))$, and  thanks to a measurable selection result [see, e.g., Corollary 5.22 of Vinalli \cite{MR2459454}], there exists a measurable mapping $(t,\be)\mapsto R_{t,\mu_{\be}}$ such that $R_{t,\mu_{\be}}\in \mathcal{P}(\rd\times\rd)$  is an optimal coupling between $f_t$ and $\mu_{\be}.$  We define 
$$
F(t,\be,B):=\frac{R_{t,\mu_{\be}}\Big(B, \{e_1\}\Big)}{R_{t,\mu_{
\be}}\Big(\rd, \{e_1\}\Big)}=\frac{nR_{t,\mu_{\be}}\Big(B, \{e_1\}\Big)}{\sharp\{l: e_l=e_1\}}=nR_{t,\mu_{\be}}\Big(B, \{e_1\}\Big),
$$
for all $t\ge 0,\ \be\in(\rd)_\bullet^n$  and any Borel set $B\subseteq\rd.$ Following from the measurability of $R_{t,\mu_{\be}}$ in $(t, \be)$, it's clear that $F: \rr_+\times(\rd)_\bullet^n\times\rd\mapsto\rr$ is a probability kernel.  Then $F$ can be associated to a measurable function $g:\rr_+\times(\rd)_\bullet^{n}\times[0,1]\mapsto\rd$ such that $g(t,\be,\alpha)$ has distribution  $F(t,\be,\cdot)$  whenever $\alpha$ is a uniform random variable in $[0,1]$ for every $(t,\be)\in\rr_+\times(\rd)_\bullet^{n}$. Now, taking $n=N-1$ for given $N\ge2$, $i\ne j,$ for any vector $\bw\in(\rd)_\bullet^{N}$, we
now define
\begin{equation}\label{cp}
\Pi_t^{ij}(\bw,\alpha)=g(t,\bw^{ij},\alpha)
\end{equation}
where $\bw^{ij}\in(\rd)_\bullet^{N-1}$ denotes the vector $\bw$ with its $i^{th}$ coordinate removed, the $j^{th}$ coordinate
in the first position, and the remaining coordinates in positions $2,...,N-1$ in an increasing order. Note \eqref{cp} and that $g(t,\be,\alpha)$ has distribution $F(t,\be,\cdot)$, we have 
\begin{align*}
    \int_0^1 |\Pi_t^{ij}(\bw,\alpha)-w_j|^2 d\alpha&=\int_0^1\big|g(t,\bw^{ij},\alpha)-w_j|^2d\alpha
    =\int_\rd |u-w_j|^2\frac{R_{t,\bar{\mu}_{\bw}^{i}}\big(du, \{w_j\}\big)}{R_{t,\bar{\mu}_{\bw}^{i}}\big(\rd, \{w_j\}\big)}\\
  &=(N-1)\int_\rd |u-w_j|^2R_{t,\bar{\mu}_{\bw}^{i}}\big(du, \{w_j\}\big).
\end{align*}
Furthermore, 
\begin{align*}
    \int_0^1 \frac{1}{N-1}\sum_{j\ne i}^N|\Pi_t^{ij}(\bw,\alpha))-w_j|^2 d\alpha=&\sum_{j\ne i}^N\int_\rd |u-w_j|^2 R_{t,\bar{\mu}_{\bw}^{i}}\big(du, \{w_j\}\big)\\
    =\int_{\rd\times\rd}|u-v|^2R_{t,\bar{\mu}_{\bw}^{i}}\big(du, dv\big)=&\cW^2_2(f_t, \bar{\mu}_{\bw}^{i}),
\end{align*}
which completes the proof of $(i)$.
\vip
We now check Point (ii). For any exchangeable random vector $\bY\in (\rd)_\bullet^N$,
recalling the definition of $\Pi^{ij}_t(\bw,\alpha),$ and for any bounded measurable function $\phi$ from $\rd$ to $\rr$, we have
\begin{align*}
    \int_0^1\phi\big(\Pi^{ij}_t(\bY,\alpha)\big)d\alpha&=\int_0^1 \phi\big(g(t,\bY^{ij},\alpha)\big)d\alpha=\int_{\rd}\phi(u)\frac{R_{t,\bar{\mu}_{\bY}^{i}}\big(du, \{Y_j\}\big)}{R_{t,\bar{\mu}_{\bY}^{i}}\big(\rd, \{Y_j\}\big)}.
\end{align*}
By the exchangeability of $\bY$, we know that $\int_{\rd}\phi(u)\frac{R_{t,\bar{\mu}_{\bY}^{i}}\big(du, \{Y_j\}\big)}{R_{t,\bar{\mu}_{\bY}^{i}}\big(\rd, \{Y_j\}\big)}$ has the same distribution for all $j\ne i$. Hence,
\begin{align*}
 \E\left[\int_0^1\phi\big(\Pi^{ij}_t(\bY,\alpha)\big)d\alpha\right]
 =\frac{1}{N-1}\sum_{j\ne i}\E\left[\int_{\rd}\phi(u)\frac{R_{t,\bar{\mu}_{\bY}^{i}}(du, \{Y_j\})}{R_{t,\bar{\mu}_{\bY}^{i}}(\rd, \{Y_j\})} \right]
=\int_{\rd}\phi(u) f_t(du).
\end{align*}
This concludes point (ii).
\end{proof}

\subsection{The particle systems}
Let's now couple the Kac's particle system with a system of cutoff Boltzmann processes. We first see the explicit formulation of  the Kac's particle system which  is implied directly by Proposition \ref{kpst}.

\begin{lem}\label{kps}
 Consider $N\ge2$, $K\in[1,\infty)$ fixed. Then there exist, on some probability space, a family of  i.i.d.  $(V_0^i)_{i=1,\dots N}$
with common law $f_0$ and an i.i.d. family of 
Poisson measures $(M_{ij}(ds,d\alpha,dz,d\varphi))_{1\le i<j\le N}$ on $[0,\infty)\times [0,1]\times[0,\infty)\times [0,2\pi)$ with intensity measure
$\frac{1}{N-1}ds d\alpha dz d\varphi$, independent of $(V_0^i)_{i=1,\dots N}$,  satisfying $M_{ij}(ds,d\alpha,dz,d\varphi)=M_{ji}(ds,d\alpha,dz,d\varphi)$ for $1\le j<i\le N$ and  $M_{ii}(ds,d\alpha,dz,d\varphi)=0$ for $i=1,\dots, N$, such that for each $i=1,\dots,N$,
the following SDEs
\beq\label{ps}
V^{i}_t=V^i_0 + \sum_{j=1}^N\intot\int_0^1\int_0^\infty\int_0^{2\pi} 
c(V^{i}_\sm,V^{j}_\sm,z,\varphi) M_{ij}(ds,d\alpha,dz,d\varphi),
\eeq
has the unique (in law) solution $(V^{i}_t)_{i=1,\dots,N, t\geq 0}$ which is a Markov process  with generator $\cL_N$ defined by \eqref{kacg}. Furthermore, it is almost surely conservative:  for all $t\ge 0$, it holds that $\sum_{i=1}^N V^{i}_t=\sum_{i=1}^N V^{i}_0 $ and $\sum_{i=1}^N|V^{i}_t|^2=\sum_{i=1}^N|V^{i}_0|^2 .$ 
\end{lem}

For $K\ge1$, according to Lemma \ref{ww2p}, we have the following formulation for the cutoff Boltzmann processes with the optimal coupling process.
\begin{lem}\label{coupling1}
For $\gamma\in(0,1]$, suppose \eqref{cs}, \eqref{c4}, \eqref{c2hs} or \eqref{c2} and let $f_0\in\cP_2(\rd)$ satisfy \eqref{c5}. Consider $N\geq 2$ and $K\ge 1$ fixed. Let  $(f_t^K)_{t\ge0}$ be the corresponding unique solution to the cutoff Boltzmann equation. $\Pi^{ij,K}(\bW^K,\alpha)$ is the measurable function introduced in Lemma \ref{ww2p}. And the Poisson measures $(M_{ij})_{i,j=1,\cdots,N}$ are introduced in Lemma \ref{kps}. Then for $i=1,\cdots,N$,
\beq\label{tbp2}
W^{i,K}_t=V^i_0 + \sum_{j=1}^N\intot\int_0^1\int_0^\infty\int_0^{2\pi} 
c_K(W^{i,K}_\sm,\Pi_s^{ij,K}(\bW^K_\sm,\alpha),z,\varphi+\varphi_{i,j,s}) M_{ij}(ds,d\alpha,dz,d\varphi),
\eeq
where $\varphi_{i,j,s}:=\varphi_0(V^{i}_\sm-V^{j}_\sm,W^{i,K}_\sm-\Pi_s^{ij,K}(\bW^K_\sm,\alpha))$,
has a  c\`{a}dl\`{a}g adapted unique strong solution $(W_t^{i,K})_{t\ge0}$. In particular, $\bW^K=(W^{1,K},\dots,W^{N,K})\in(\rd)^N_\bullet$, and  for each $t\geq 0$, $(W^{i,K}_t)_{i=1,\dots,N}$ are  with common law $f_t^K$. 
\end{lem}
\begin{proof}
The  existence of $\{W^{i,K}_t\}_{t\ge0}$  for $i=1,...,N$ to \eqref{tbp2} is obvious. To  obtain the uniqueness in law, it is sufficient to show that  for fixed $i=1,...,N$, $\{W^{i,K}_t\}_{t\ge0}$ is equivalent to $(W_t^K)_{t\ge0}$ determined by \eqref{cbp} in law in  Remark \ref{cutbp}. We define a family of random measure $(Q_i(ds,d\xi,dz,d\varphi))_{1\le i\le N}$ on 
$[0,\infty)\times [0,N]\times[0,\infty)\times [0,2\pi)$. For any measurable set $A_1\subseteq [0,\infty),$ $A_2\subseteq [0,N],$ $A_3\subseteq [0,\infty),$ $A_4\subseteq [0,2\pi),$
\begin{align*}
    Q_i(A_1,A_2,A_3,A_4)= \sum_{j=1}^N M_{ij}(A_1,j-(A_2\cap(j-1,j]),A_3,A_4).
\end{align*}
It's not hard to see that $ Q_i(ds,d\xi,dz,d\varphi)$ is Poisson measure on $[0,\infty)\times [0,N]\times[0,\infty)\times [0,2\pi)$ with intensity
$ds dz d\varphi d\xi\indiq_{D^i}(\xi)/(N-1)$, independent of $(V_0^i)_{i=1,\dots N}$, where $D^i:=[0,N]\setminus[i-1,i)$. Then we can rewrite
\begin{align*}
W^{i,K}_t=&V^i_0 + \intot\int_0^N\int_0^\infty\int_0^{2\pi} 
c(W^{i,K}_\sm,\Pi_s^{i}(\bW^K_\sm,\xi),z,\varphi+\varphi_{i,j,s}) Q_i(ds,d\xi,dz,d\varphi),
\end{align*}
where $\Pi_s^{i}(\bW^K_\sm,\xi)=\Pi_s^{i\lfloor\xi\rfloor}(\bW^K_\sm,\xi-\lfloor\xi\rfloor)$. Next, we define $Q^{*}_i(ds,dv,dz,d\varphi)$ to be the point measure on $[0,\infty)\times \rd\times[0,\infty)\times [0,2\pi)$ with atoms $(s,\Pi_s^{i\lfloor\xi\rfloor}(\bW^K_\sm,\xi-\lfloor\xi\rfloor),z,\varphi),$ which means,
for any measurable set $B\subseteq[0,\infty)\times \rd\times[0,\infty)\times [0,2\pi)$, that
\begin{align*}
   Q_i^{*}(B):=  Q_i\Big(\big\{(s,\xi,z,\varphi)|(s,\Pi_s^{i}(\bW^K_\sm,\xi),z,\varphi)\in B\big\}\Big).
\end{align*}
 It's clear that $ W^{i,K}_t$ is only determined by $\{Q^{*}_i(ds,dv,dz,d\varphi)\}$ and $V_0^i$. Since the $\xi-$law of $\Pi_t^{i}(\bW^K_\sm,\xi)$ is $f_t^K$ for $t\ge0$,  $\varphi_{i,j,s}\in[0,2\pi)$ and $c_K$ is a periodic function with a period of $2\pi$.  Thus we can say that  $Q^{*}_i(ds,dv,dz,d\varphi)$ is a Poisson measures with the same intensity $dsf_s^K(dv)dzd\varphi$ by computing the corresponding Laplace functional using It\^{o}'s formula. Hence, for $i=1,...,N$ and each $t\ge0$, $ W^{i,K}_t$ is a cutoff Boltzmann process with a unique common law $f_t^K$.
 \end{proof}
 
\begin{rem}
Since $(V^i_0)_{i=1,\dots,N}$ and $(M_{ij})_{i,j=1,\dots,N}$ are exchangeable and  for each $i=1,...,N$, $W_t^{i,K}$  and $V_t^{i}$ are both unique in law, the family $\Big((W^{1,K}_t,V^{1}_t)_{t\geq 0},...,(W^{N,K}_t,V^{N}_t)_{t\geq 0}\Big)$ is exchangeable.
\end{rem}

\section{The Second coupling}
\setcounter{equation}{0} 
In this section, we will prove our main results  using the second coupling method introduced in \cite{MR3476628,MR3769742,MR3621429} where the authors deal with the Kac's model for the Boltzmann equation and the Landau equation respectively.
We first introduce some more notations. We define 
\begin{equation}\label{bestrate}
\e^{K}_t(f_t^K):= \E \left[\cW_2^2\left(f_t^K,\frac{1}{N-1}\sum_{i=2}^N \delta_{W^{i,K}_t} \right)\right]
\end{equation}
where $(W^{i,K}_t)_{i=1,...,N, t\ge0}$ were defined in Lemma \ref{coupling1}. By Lemma \ref{ww2p} and the exchangeability, it's clear that for all $i=1,\cdots,N$,
\begin{align*}
  \E \left[  \int_0^1 \frac{1}{N-1}\sum_{j\ne i}^N|\Pi_t^{ij,K}(\bW^K_t,\alpha))-W^{j,K}_t|^2 d\alpha \right]=\e^{K}_t(f_t^K).
\end{align*}
We also denote that 
\begin{equation}\label{bestrate1}
\mu^{N,K}_{\bW^K_t}:= \frac{1}{N}\sum_{i=1}^N \delta_{W^{i,K}_t}.
\end{equation}

\subsection{Estimate of the Wasserstein distance}
We can now prove the following important intermediate result. 

\begin{prop}\label{mrl} 
For $\gamma\in(0,1]$, suppose the collision kernel  $B$ satisfying \eqref{cs}, \eqref{c4}, \eqref{c2hs} or \eqref{c2}. Let $f_0\in\cP_2(\rd)$ not be a Dirac mass and satisfy \eqref{c5}. For each $K\ge1$, consider the unique cutoff solution $(f_t^K)_{t\geq 0}$ introduced  in Remark \ref{cutws} and for each $N\ge1$, the unique Markov process $(V^{i}_t)_{i=1,\dots,N,t\geq 0}$  defined in Lemma \ref{kps}.
Set $\mu^{N}_t:=N^{-1}\sum\limits_{i=1}^{N}\delta_{V^{i}_t}$.

(i) Hard potentials. Assume that \eqref{cs}, \eqref{c2} with $\gamma\in(0,1)$. For all $\e\in(0,1)$, all $T\geq0$, 
there is a constant $C_{\e,T}$ such that for all $N\geq 1$, all $K\in[1,\infty]$,
\begin{align}\label{fc31}
\sup_{[0,T]}\E[\cW_2^2(\mu^{N}_t,f_t)] \leq& C_{\e,T} \left(\sup_{[0,T]}\e^K_t(f_t^K)
+ K^{1-2/\nu} \right)^{1-\e}+\frac{C}{N}.
\end{align}

(ii) Hard spheres. Assume finally that $\gamma=1$, \eqref{cs} and  \eqref{c2hs} and that $f_0$ has a density.
For all $\e\in(0,1)$, all $T\geq0$, all $q\in (1,p_0)$, 
there is a constant $C_{\e,q,T}$ such that for all $N\geq 1$, all $K\in[1,\infty)$,
\begin{align}\label{fc51}
\sup_{[0,T]}\E[\cW_2^2(\mu^{N}_t,f_t)] \leq& C_{\e,q,T} \left( K^{2\e} \left(\sup_{[0,T]}\e^K_t(f_t^K)\right)^{1-\e/2}
+ K^{1-2/\nu} \right)+\frac{C}{N} .
\end{align}

\end{prop}
We first introduce a lemma which will be needed to prove Proposition \ref{mrl}.

\begin{lem} \label{tem}
For $K\ge1$, recall that $\e^{K}_t(f_t^K)$  and $\mu^{N,K}_{\bW^K_t}$ defined in $(\ref{bestrate})$ and $(\ref{bestrate1}),$ then we have 
\begin{align*}
    \E[\cW_2^2(\mu^{N,K}_{\bW^K_t}, f_t^K)]\le \frac{N-1}{N}\e^{K}_t(f_t^K)+\frac{C}{N}.
\end{align*}
\end{lem}
\begin{proof}
It  follows mainly from  \cite[Lemma 2.1-step 1]{MR3621429}. Reecall that for $f,f',g,g'\in\cP_2(\rd)$ and $\lambda\in(0,1)$, it holds that
\beq\label{wline}
\cW_2^2(\lambda f + (1-\lambda) g,\lambda f' + (1-\lambda) g')\leq \lambda \cW_2^2(f,f')+(1-\lambda)\cW_2^2(g,g').
\eeq
Indeed, consider $X\sim f$ and $X'\sim f'$ satisfying $\E[|X-X'|^2]=\cW_2^2(f,f')$,  $Y\sim g$ and $Y'\sim g'$ 
satisfying  $\E[|Y-Y'|^2]=\cW_2^2(g,g')$, and a Bernoulli random variable $U\sim$ Bernoulli$(\lambda)$, with $(X,X'),(Y,Y'),U$ independent. Set $Z:=UX+(1-U)Y\sim \lambda f + (1-\lambda) g$, $Z':=UX'+(1-U)Y'\sim \lambda f' + (1-\lambda) g'$.
It's easy to verify that $\E[|Z-Z'|^2]=\lambda\E[|X-X'|^2]+(1-\lambda)\E[|Y-Y'|^2]
=\lambda \cW_2^2(f,f')+(1-\lambda)\cW_2^2(g,g')$.
Thus, we have
\begin{align*}
    \E[\cW_2^2(\mu^{N,K}_{\bW^K_t}, f_t^K)]&=\E\left[\cW^2_2\Big(\frac{\delta_{W^{1,K}_t}}{N}+\frac{N-1}{N}\frac{1}{N-1}\sum_{i=2}^N \delta_{W^{i,K}_t},\frac{f_t^K}{N}+\frac{N-1}{N}f_t^K\Big)\right]\\
   & \le \frac{1}{N}\E\left[\cW^2_2(\delta_{W^{1,K}_t},f_t^K)\right]+\frac{N-1}{N}\E\left[\cW^2_2\Big(\frac{1}{N-1}\sum_{i=2}^N \delta_{W^{i,K}_t},f_t^K\Big)\right]\\
   &\le \frac{C}{N}+\frac{N-1}{N}\e^{K}_t(f_t^K).
\end{align*}
 The last inequality attributes to the fact that  $\E\left[\cW^2_2(\delta_{W^{1,K}_t},f_t^K)\right]\le 2\E[|W_t^{1,K}|^2]+2\int_\rd |v|^2f_t^K(dv)=4\int_\rd |v|^2f_t^K(dv)=4\int_\rd |v|^2f_0(dv)<\infty$.
\end{proof}

We now move to see

\begin{proof}[Proof of Proposition \ref{mrl}-(i)]  Assume $\gamma\in(0,1)$.
Consider $f_0\in \cP_2(\rd)$ satisfying \eqref{c5}.
We also assume that $f_0$ is not a Dirac mass, so that $f_t$ has a density for all $t>0$.
Consider $N\geq 1$ and $K\in [1,\infty)$ fixed and 
the coupled processes introduced in Lemma \ref{kps} and  Lemma  \ref{coupling1}.

\vip

{\it Step} 1. Using the It\^o's formula, we  have
\begin{align*}
&\E[|V^{1}_t-W^{1,K}_t|^2]\\
=& \frac{1}{N-1}\sum_{j=2}^N\intot \int_0^1 \int_0^\infty \int_0^{2\pi} \E\Big[ |V^{1}_s-W^{1,K}_s + \Delta^{1,j}(s,\alpha,z,\varphi) |^2 - 
 |V^{1}_s-W^{1,K}_s|^2 \Big] d\varphi dz d\alpha ds,
\end{align*}
where
\begin{align*}
\Delta^{1,j}(s,\alpha,z,\varphi)=c(V^{1}_s,V_s^{j},z,\varphi)-c_K(W^{1,K}_s,\Pi_s^{1j,K}(\bW^K_\sm,\alpha),z,\varphi+\varphi_{1,j,s}).
\end{align*}
Using Lemma \ref{fundest}-(i) with $v=V_s^1, v_*=V_s^j, \tv=W_s^{1,K}, \tv_*=\Pi_s^{1j,K}(\bW^K_\sm,\alpha)$,  we have
\begin{align*}
\E[|V^{1}_t-W^{1,K}_t|^2] \leq  \intot [B_1^K(s)+B_2^K(s)+B_3^K(s) ]ds,
\end{align*}
where for $i=1,2,3$,
\begin{align*}
&B_i^K(s):= \frac{1}{N-1} \sum^N_{j=2}\int_0^1 \E \Big[I_i^K\big(V_s^1,V_s^{j},W^{1,K}_s,\Pi_s^{1j,K}(\bW^K_\sm,\alpha)\big)\Big] d\alpha.
\end{align*}

\vip
{\it Step} 2. Using Lemma \ref{further}-(i),
we see that for all $M\geq 1$ and fixed $q\in(\gamma, p_0)$,
\begin{align*}
B_1^K(s) \leq& \frac{M}{N-1} \sum_{j=2}^N\int_0^1 \E\left[ |V^{1}_s-W^{1,K}_s|^2+ |V_s^{j} -
\Pi_s^{1j,K}(\bW^K_\sm,\alpha)|^2 \right] d\alpha \\
&+  \frac{Ce^{-M^{q/\gamma}}}{N-1} \sum_{j=2}^N\E \left[ \int_0^1 \exp{\left(C\Big(|W^{1,K}_s|^{q}+|\Pi_s^{1j,K}(\bW^K_\sm,\alpha)|^{q} \Big)\right) }d\alpha \right] \\
\leq & \frac{M}{N-1} \sum_{j=2}^N\int_0^1 \E\left[ |V^{1}_s-W^{1,K}_s|^2+ |V_s^{j} -
\Pi_s^{1j,K}(\bW^K_\sm,\alpha)|^2 \right] d\alpha + Ce^{-M^{q/\gamma}}.
\end{align*}
To obtain the last inequality, we use  Lemma \ref{ww2p}-(ii) and that $W^{1,K}_s$ is $f_s^K$-distributed, whence for any $j=2,...,N$,
\begin{align}\label{nones}
&\E \left[ \int_0^1 \exp{\Big(C\Big(|W^{1,K}_s|^{q}+|\Pi_s^{1j,K}(\bW^K_\sm,\alpha)|^{q} \Big)\Big) } d\alpha \right]\notag \\
&\le \E \left[ \int_0^1 \exp{\Big(2C|W^{1,K}_s|^{q}\Big) } d\alpha\right]^{1/2} \E \left[ \int_0^1 \exp{\Big(2C|\Pi_s^{1j,K}(\bW^K_\sm,\alpha)|^{q} \Big)}  d\alpha\right]^{1/2} \\
&=\int_\rd e^{2C|v|^q}f_s^K(dv)<\infty \notag,
\end{align}
by  \eqref{momex}.

\vip

{\it Step} 3. In this step, we will study $B^K_2$.  We introduce 
$$
\tB^K_2(s):=\frac{1}{N-1}\sum_{j=2}^N\int_0^1\E\left[ I_2^K(V^{1}_s,V^{j}_s,W^{1,K}_s, W^{j,K}_s) \right] d\alpha
$$
and observe that  $\tB^K_2(s)=0$. Indeed,
\begin{align*}
\tB^K_2(s)= \E\left[ I_2^K(V^{1}_s,V^{2}_s,W^{1,K}_s, W^{2,K}_s) \right]
\end{align*}
by exchangeability.
And using again exchangeability, we have
\begin{align*}
\tB^K_2(s)=\frac{1}{2}\E\left[ I_2^K(V^{1}_s,V^{2}_s,W^{1,K}_s, W^{2,K}_s, 
) \right]
+ \frac{1}{2}\E\left[ I_2^K(V^{2}_s,V^{1}_s,W^{2,K}_s, W^{1,K}_s) \right].
\end{align*}
The symmetry of $I_2^K$, i.e.  $I_2^K(v,v_*,\tv,\tv_*)+I_2^K(v_*,v,\tv_*,\tv)=0$ implies $\tB^K_2(s)=0$. 
Whence, 
\begin{align*}
B_2^K(s)= &\frac{1}{N-1}\sum^N_{j=2}\int_0^1 \E\Big[I_2^K(V^{1}_s, V^{j}_s, W^{1,K}_s,\Pi_s^{1j,K}(\bW^K_\sm,\alpha))  \\
& \hskip4cm-I_2^K(V^{1}_s,V^{j}_s,W^{1,K}_s, W^{j,K}_s)  \Big] d\alpha.
\end{align*}
Consequently, according to Lemma \ref{further}-(ii) with $z_*=W_s^j$, we have
\begin{align*}
B_2^K(s)\leq & \frac{C}{N-1}\sum_{j=2}^N\int_0^1 \E\Big[ |V^{1}_s- W^{1,K}_s|^2+ |V_s^{j} - \Pi_s^{1j,K}(\bW^K_\sm,\alpha)|^2\\
&\hskip2.5cm  + |\Pi_s^{1j,K}(\bW^K_\sm,\alpha)- W^{j,K}_s|^2\\
&\hskip 2.7cm\times (1+|V^{1}_s|+|V^{j}_s|+ |W^{1,K}_s|+|\Pi_s^{1j,K}(\bW^K_\sm,\alpha)|+|W^{j,K}_s| )^2\Big] d\alpha. 
\end{align*}

{\it Step} 4. Finally, using Lemma \ref{further}-(iii), we obtain  for any fixed $p\ge 2+2\gamma/\nu$, 
\begin{align*}
B_3^K(s) \leq& \frac{C_{p} K^{1-2/\nu}}{N-1}\sum_{j=2}^N\int_0^1 \E\Big[1+ |V^{1}_s|^{p}  + |V^{j}_s|^{p}  + |W^{1,K}_s|^{p}  \Big] d\alpha\\
=&C_{p}  K^{1-2/\nu}\int_0^1 \E\Big[1+ |V^{1}_s|^{p}  + |V^{2}_s|^{p}  + |W^{1,K}_s|^{p}  \Big] d\alpha
\end{align*}
Since $W^{1,K}_s\sim f^K_s$, we deduce from \eqref{momex} that $\E[|W^{1,K}_s|^{p} ]= 
\intrd |v|^{p}  f^K_s(dv) < \int_{\rd} e^{C_q|v|^q}f^K_s(dv)< \infty$.
The exchangeability and Proposition \ref{momentest} shows that for all $i=1,...,N$, $\E[|V^{i}_s|^{p} ] \le C_{p} \E[|V^{i}_0|^{p} ] <\infty$ by \eqref{c5}. As a result,
\[B_3^K(s) \le CK^{1-2/\nu} (\text{$C$ depending on $p_0,q,\gamma,\nu$ }).\]

\vip

{\it Step} 5. We set $h_t:=\E[ |V^{1}_t- W^{1,K}_t|^2]$. Using the above steps, we see that
for all $M\geq 1$,
\begin{align}\label{hest}
 h_t &\leq Ct e^{-M^{q/\gamma}} +Ct K^{1-2/\nu}\notag\\
&\hskip0.5cm +  (M+C)\intot \Big[ h_s 
+ \frac{1}{N-1}\sum_{j=2}^N\int_0^1 \E [|V_s^{j}-\Pi_s^{1j,K}(\bW^K_\sm,\alpha)|^2  ] d\alpha \Big]ds \notag\\
&\hskip0.5cm +\frac{C}{(N-1)}\sum_{j=2}^N \intot   \int_0^1 \E\Big[|\Pi_s^{1j,K}(\bW^K_\sm,\alpha)- W^{j,K}_s|^2\\\notag
&\hskip 2cm\times (1+|V^{1}_s|+|V^{j}_s|+ |W^{1,K}_s| + |\Pi_s^{1j,K}(\bW^K_\sm,\alpha)|+|W^{j,K}_s|  )^2 \Big] d\alpha  ds. \notag
\end{align}
We now write
\begin{align*}
&\int_0^1 \E\Big[|V_s^{j}-\Pi_s^{1j,K}(\bW^K_\sm,\alpha)|^2\Big] d\alpha\\
\leq & 2\int_0^1 \E\Big[|V_s^{j}-W^{j,K}_s|^2\Big] d\alpha 
+2  \int_0^1 \E\Big[|W^{j,K}_s-\Pi_s^{1j,K}(\bW^K_\sm,\alpha)|^2\Big] d\alpha . 
\end{align*}
Then, we deduce from exchangeability and from Lemma  \ref{ww2p}-(i) that
\begin{align}\label{topcool3}
&\frac{1}{N-1}\sum_{j=2}^N\int_0^1 \E\Big[|V_s^{j}-\Pi_s^{1j,K}(\bW^K_\sm,\alpha)|^2\Big] d\alpha \notag \\
\leq& \frac{2}{N-1}\sum_{j=2}^N\int_0^1 \E\Big[|V_s^{j}-W^{j,K}_s|^2\Big] d\alpha  + \frac{2}{N-1}\sum_{j=2}^N\int_0^1 \E\Big[|W^{j,K}_s-\Pi_s^{1j,K}(\bW^K_\sm,\alpha)|^2\Big] d\alpha \\
=&2h_s+2\e^{K}_s(f^K_s).\notag
\end{align}
Next, we bound the last term of \eqref{hest}. Applying the H\"{o}lder inequality for all $\e \in(0,1)$,
\begin{align*}
&\int_0^1 \E\Big[|\Pi_s^{1j,K}(\bW^K_\sm,\alpha)- W^{j,K}_s|^2\\
&\hskip2cm
\times \big (1+|V^{1}_s|+|V^{j}_s|+|W^{1,K}_s|+ |\Pi_s^{1j,K}(\bW^K_\sm,\alpha)|+|W^{j,K}_s|  \big)^2 \Big] d\alpha   \\
\leq& \int_0^1 \E\Big[|\Pi_s^{1j,K}(\bW^K_\sm,\alpha)- W^{j,K}_s|^{2-\varepsilon}\notag\\
&\hskip2cm\times
\big (1+|V^{1}_s|+|V^{j}_s| +|W^{1,K}_s|+ |\Pi_s^{1j,K}(\bW^K_\sm,\alpha)|+|W^{j,K}_s|  \big)^{2+\varepsilon} \Big] d\alpha    \\
\leq& \left(\int_0^1 \E\Big[|\Pi_s^{1j}(\bW_\sm,\alpha)- W^{j}_s|^{2}\Big]d\alpha\right)^{(2-\e)/2} \notag\\
&\hskip2cm \times
\left(\int_0^1 \E\Big[ \big (1+|V^{1}_s|+|V^{j}_s|+ |W^{1,K}_s|+ |\Pi_s^{1j,K}(\bW^K_\sm,\alpha)|+|W^{j,K}_s|  \big)^{(4/\e)+2} \Big] d\alpha \right)^{\varepsilon/2}\nonumber.
\end{align*}
Using again the  H\"{o}lder inequality, we have 
\begin{align}\label{topcool2}
&\frac{1}{N-1}\sum_{j=2}^N\int_0^1 \E\Big[|\Pi_s^{1j,K}(\bW^K_\sm,\alpha)- W^{j,K}_s|^2  \notag \\
&\hskip2cm
\times \big (1+|V^{1}_s|+|V^{j}_s|+|W^{1,K}_s|+ |\Pi_s^{1j,K}(\bW^K_\sm,\alpha)|+|W^{j,K}_s|  \big)^2 \Big] d\alpha  \notag \\
\le& \left(\frac{1}{N-1}\sum_{j=2}^N\int_0^1 \E\Big[|\Pi_s^{1j}(\bW_\sm,\alpha)- W^{j}_s|^{2}\Big]d\alpha\right)^{(2-\e)/2} \\
& \times
\left(\frac{1}{N-1}\sum_{j=2}^N \int_0^1 \E\Big[ \big(1+|V^{1}_s|+|V^{j}_s|+|W^{1,K}_s|+ |\Pi_s^{1j,K}(\bW^K_\sm,\alpha)|+|W^{j,K}_s|  \big)^{(4/\varepsilon)+2} \Big] d\alpha \right)^{\e/2} \notag \\
\le & C_\e\left( \e^{K}_s(f^K_s) \right)^{(2-\e)/2} \notag.
\end{align}
This last inequality concluded by Proposition \ref{momentest} and  \eqref{c5},  
\begin{align*}
\E\left[|V^{1}_s|^{(4/\varepsilon)+2}  \right] =\E\left[|V^{j}_s|^{(4/\varepsilon)+2} \right]
\le &C_{\e} \E\left[|V^{1}_0|^{(4/\varepsilon)+2} \right]
\leq C_\e.
\end{align*}
and by  Lemma \ref{ww2p}-(ii), that $W_s^{j,K}\sim f_s^K$ and \eqref{momex}, 
\begin{align*}
 \int_0^1 \E |\Pi_s^{1j,K}(\bW^K_\sm,\alpha)|^{(4/\varepsilon)+2} d\alpha=\E|W_s^{j,K}|^{(4/\varepsilon)+2}\le C_\e.
\end{align*}
Therefore, we have for all $\e\in(0,1)$, all $M\geq 1$,
\begin{align*}
h_t \leq& Ct  e^{-M^{q/\gamma}} +Ct K^{1-2/\nu}+3(M+C)\intot \big[ h_s +\e^{K}_s(f_s^K)
\big] ds \\
&+C_\e \intot ( \e^{K}_s(f_s^K) )^{1-\e/2} ds.
\end{align*}
Recall \eqref{bestrate}, and that 
$W^{1,K}_t,\dots,W^{N,K}_t$ have the common law $f_t^K$, we  thus have $\e_s^K(f_t^K)\leq 2 \intrd |v|^2f_t^K(dv)=2\intrd |v|^2f_0(dv)$. Since $M\geq 1$, $K\in [1,\infty)$,
we observe 
\begin{align*}
h_t \leq Ct e^{-M^{q/\gamma}} +CMt (K^{1-2/\nu})^{1-\e/2}
 +  CM\intot h_s ds + C_\e M \intot ( \e^{K}_s(f_s^K) )^{1-\e/2} ds.
\end{align*}
Set 
\begin{align*}
\xi_{K,t} := K^{1-2/\nu}+ \sup_{[0,t]} \e^{K}_s(f_s^K),
\end{align*}
 we thus get
\begin{align*}
h_t \leq& C_\e \left( t e^{-M^{q/\gamma}} + M t\xi_{K,t}^{1-\e/2}  +  M \intot h_s ds\right).
\end{align*}
We deduce from the  Gr\"onwall's lemma that, for any $M\geq 1$,
\begin{align*}
\sup_{[0,T]} h_t \leq& C_\e T\left(  e^{-M^{q/\gamma}} + M \xi_{K,T}^{1-\e/2}\right) e^{C_\e M T}.
\end{align*}
 Furthermore, we conclude that
\begin{align*}
\sup_{[0,T]} h_t \leq& C_{\e,T} \xi_{K,T}^{1-\e}
\end{align*}
by taking $M=1$  if $\xi_{K,T} \geq 1/e$
and $M=|\ln \xi_{K,T}|^{\gamma/q}$ otherwise. More precisely, when $\xi_{K,T} < 1/e$ and $M=|\ln \xi_{K,T}|^{\gamma/q}$, due to $\gamma/q<1$ (Noting fixed $q\in(\gamma,p_0)$ in step 2), we obtain
\begin{align*}
\sup_{[0,T]} h_t \leq& C_\e T \left( \xi_{K,T} + \xi_{K,T}^{1-\e/2} |\ln \xi_{K,T}|^{\gamma/q}    
\right) e^{C_\e |\ln \xi_{K,T}|^{\gamma/q} T}\leq  C_{\e,T} \xi_{K,T}^{1-\e}.
\end{align*}

{\it Final Step.} Recall that $\mu^{N}_t:=N^{-1}\sum\limits_{i=1}^{N}\delta_{V^{i}_t}$, 
$$
\E[\cW_2^2(\mu^{N}_t,f_t)]\leq 2 \E[\cW_2^2(\mu^N_t,\mu^{N,K}_{\bW_t^K})]
+ 4 \E[\cW_2^2(\mu^{N,K}_{\bW_t^K},f_t^K)]+ 4 \E[\cW_2^2(f_t^K,f_t)].
$$
But $\E[\cW_2^2(\mu^N_t,\mu^{N,K}_{\bW_t^K})]\leq \E[N^{-1}\sum_{i=1}^N |V^{i}_t-W_t^{i,K}|^2] =
\E[|V^{1}_t-W_t^{1,K}|^2]=h_t$ by exchangeability,
and we have already seen that $\E[\cW_2^2(\mu^{N,K}_{\bW_t^K},f_t^K)]\le \frac{N-1}{N}\e^{K}_t(f_t^K)+\frac{C}{N}$ from Lemma \ref{tem} and by \eqref{cutf}. Consequently,
for all $\e\in(0,1)$, all $t\in [0,T]$,
\begin{align*}
\E[\cW_2^2(\mu^{N}_t,f_t)]&\leq C_{\e,T}\xi_{K,T}^{1-\e} + 4\e^{K}_t(f_t^K)+\frac{C}{N}+ctK^{1-2/\nu}\\
&\leq C_{\e,T} \left(  K^{1-2/\nu}+  \sup_{[0,T]} \e^{K}_s(f_s^K) \right)^{1-\e}+\frac{C}{N},
\end{align*}
and this proves \eqref{fc31}. 
\end{proof}

We now study the hard spheres.

\begin{preuve} {\it of Proposition \ref{mrl}-(ii).}
For $\gamma=1$, we assume \eqref{cs} and \eqref{c2hs}. We consider $f_0\in \cP_2(\rd)$ satisfying \eqref{c5} for some $p_0\in(\gamma,2)$ and fix $q\in(\gamma,p_0)$ 
for the rest of the proof. We also assume that $f_0$ has a density, so that $f_t$ has a density for all $t>0$.
We fix $N\geq 1$ and $K\in [1,\infty)$ and consider 
the processes introduced in Lemma \ref{kps} and Lemma \ref{coupling1}. 
\vip
 We write similarly to the case of hard potentials and recall $h_t=\E[|V^{1}_t -W^{1,K}_t |^2]$,  then
\begin{align*}
h_t\leq \intot [B_1^K(s)+B_2^K(s)+B_3^K(s) ]ds,
\end{align*}
where  $B_i^K(s):= \frac{1}{N-1} \sum^N_{j=2}\int_0^1 \E \big[I_i^K\big(V_s^1,V_s^{j},W^{1,K}_s,\Pi_s^{1j,K}(\bW^K_\sm,\alpha)\big)\big] d\alpha$
 for $i=1,2,3.$ Using Lemma \ref{furtherhs} and following exactly the proof of hard potentials, we deduce that for all $M>1$ and any fixed $q\in(1, p_0)$,
\begin{align*}
    B^K_1(s)+B^K_2(s)&\le \frac{2M}{N-1} \sum_{j=2}^N \int_0^1 \E\Big[\abs{V_s^1-W_s^{1,K}}^2+\abs{V_s^j-\Pi_s^{1j,K}(\boldsymbol{W}_{s-},\alpha)}^2\Big] d\alpha\\
     &+  \frac{C}{N-1} \sum_{j=2}^N \int_0^1 \E\Big[ \abs{\Pi_s^{1j,K}(\boldsymbol{W}_{s-},\alpha)-W_s^{j,K}}^2\\
    &\hskip5cm \times\Big(1+\abs{W_s^{1,K}}+\abs{\Pi_s^{1j,K}(\boldsymbol{W}_{s-})}+\abs{W_s^{j,K}}\Big)^{2}\Big]d\alpha\\
     &+\frac{C K e^{-M^q}}{N-1} \sum_{j=2}^N \int_0^1 \E\Big[ \Big(1+\abs{V_s^j}+\abs{V_s^1}\Big) \\
  &   \hskip3.5cm \times \exp \left\{C\Big(\abs{W_s^{1,K}}^q+\abs{\Pi_s^{1j,K}(\boldsymbol{W}_{s-},\alpha)}^q+\abs{W_s^{j,K}}^q\Big)\right\}\Big]d\alpha.
\end{align*}
and for any $p>0$,
\begin{align*}
    B^K_3(s)\le \frac{C_{p} K^{-p}}{N-1} \sum_{j=2}^N \int_0^1 \E\Big[1+\abs{V_s^1}^{p+3}+\abs{V_s^j}^{p+3}+\abs{W^{1,K}_s}^{p+3}\Big]\le \frac{C_{p}}{K^{p}},
\end{align*}
due to Proposition  \ref{momentest}, \eqref{c5} and \eqref{momex}.
We also note that 
\begin{align*}
   & \int_0^1 \E\Big[ \Big(1+\abs{V_s^j}+\abs{V_s^1}\Big) \exp \left\{C\Big(\abs{W_s^{1,K}}^q+\abs{\Pi_s^{1j,K}(\boldsymbol{W}_{s-},\alpha)}^q+\abs{W_s^{j,K}}^q\Big)\right\}\Big]d\alpha\\
    &\le   \int_0^1 \E\Big[ \Big(1+\abs{V_s^j}+\abs{V_s^1}\Big)^2 +\exp  \left\{2 C\Big(\abs{W_s^{1,K}}^q+\abs{\Pi_s^{1j,K}(\boldsymbol{W}_{s-},\alpha)}^q+\abs{W_s^{j,K}}^q\Big) \right\}\Big]d\alpha \le C,
\end{align*}
owing to the fact that $\E|V_s^j|^2=\E|V_0^1|^2=\int_{\rd}|v|^2f_0(dv)<\infty$ for $j=1,2,\cdots,N$ and that  a similar computation  to \eqref{nones}. Proceeding as in \eqref{topcool2}, it holds that  for any $\e\in(0,1),$ 
\begin{align*}
  &\frac{1}{N-1}\sum_{j=2}^N  \int_0^1 \E\Big[ \abs{\Pi_s^{1j,K}(\boldsymbol{W}_{s-},\alpha)-W_s^{j,K}}^2\Big(1+\abs{W_s^{1,K}}+\abs{\Pi_s^{1j,K}(\boldsymbol{W}_{s-})}+\abs{W_s^{j,K}}\Big)^{2}\Big]d\alpha\\
  &\le C_\e\big(\e_s^K(f_s^K)\big)^{1-\frac{\e}{2}}.
\end{align*}
Using \eqref{topcool3}, we thus have
\begin{align*}
     B^K_1(s)+B^K_2(s)+B^K_3(s)\le& 6Mh_s+ 4M\e_s^K(f_s^K)+C_\e\big(\e_s^K(f_s^K)\big)^{1-\frac{\e}{2}}+C K e^{-M^q}+\frac{C_{p}}{K^{p}}.
\end{align*}
Hence, for $0\le t\le T$, any $M>1$, any $\e\in(0,1)$,
\begin{align*}
    h_t \le& 6M\int_0^t h_s ds +CtK e^{-M^q}+C_{p} tK^{-p}+4M \int_0^t \e_s^K(f_s^K) ds+  C_\e\int_0^t \big(\e_s^K(f_s^K)\big)^{1-\frac{\e}{2}} ds\\
    \le&  6M\int_0^t h_s ds +CtK e^{-M^q}+C_{p} tK^{-p}+  C_\e M\int_0^t \big(\e_s^K(f_s^K)\big)^{1-\frac{\e}{2}}ds\\
    \le& C_\e\Big( M\int_0^t h_s ds+K e^{-M^q}T +TK^{-p}+  M\int_0^T \big(\e_s^K(f_s^K)\big)^{1-\frac{\e}{2}}ds\Big).
\end{align*}

 By the Gr\"onwall's inequality, we deduce that 
\begin{align*}
   \sup_{t\in [0,T]} h_t \le C_\e T\Big(K^{-p}+K e^{-M^q}+M\sup_{t\in[0,T]}\big(\e_t^K(f_t^K)\big)^{1-\frac{\e}{2}}\Big) e^{C_\e MT}.
\end{align*}
Choosing $M=\e(C_\e T)^{-1}\log(K)+1$ and recalling that $q>1,$ we deduce
\begin{align*}
     \sup_{t\in [0,T]} h_t \le C_{\e ,T}\Big( K^{\e-p}+K^{1+\e} e^{-\e^q(C_\e T)^{-q} \log^q(K)}+K^\e\big(1+\log(K)\big)\sup_{t\in[0,T]}\big(\e_t^K(f_t^K)\big)^{1-\frac{\e}{2}}\Big).
\end{align*}
We observe that there exists some positive constant $C_{\e,T},$ such that $K^{1+\e} e^{-\e^q(C_\e T)^{-q} \log^q(K)}\le C_{\e,T}  K^{\e-p}$. Indeed, it suffices to show that for any $m>0,$ there exists some $C_m>0,$ such that $ e^{-m \log^q(K)}= K^{-m\log^{q-1}(K)}\le C_m  K^{-p-1}$. In fact, if $m\log^{q-1}(K)\ge p+1,$ we just choose $C_m=1.$ Otherwise,  it implies that  $\log(K)\le \Big(\frac{p+1}{m}\Big)^{1/(q-1)}$, i.e.
\begin{align*}
    K \le \exp\Big( (p+1)^{1/(q-1)}/m^{1/(q-1)}\Big):=c_m.
\end{align*}
 Set $C_m=c_m^{p+1}$,  we thus have $K^{p+1}K^{-m\log^{q-1}(K)}\le K^{p+1}\le c_m^{p+1}=C_m. $
Moreover, we notice that $1+\log(K)\le \e^{-1} K^\e$ for any $K\ge 1$. Consequently, letting  $p=\e+ 2/\nu-1,$ we  have 
\begin{align*}
     \sup_{t\in [0,T]} h_t \le&  C_{\e ,T}\Big( K^{1-2/\nu}+K^{2\e}\sup_{t\in[0,T]}\big(\e_t^K(f_t^K)\big)^{1-\frac{\e}{2}}\Big).
\end{align*}
Operating as the final step in hard potentials,  using that  
$
\E[\cW_2^2(\mu^{N}_t,f_t)]\leq 2 h_t
+ 4 \E[\cW_2^2(\mu^{N,K}_{\bW_t^K},f_t^K)]+ 4 \E[\cW_2^2(f_t^K,f_t)],$ we finally conclude that   for any $K\ge 1,$ $\e\in(0,1),$
\begin{align*}
\sup_{[0,T]}\E[\cW_2^2(\mu^N_t,f_t)] \leq C_{\e ,T}\left( K^{1-2/\nu}+K^{2\e}\Big(\sup_{t\in[0,T]}\e_t^K(f_t^K)\Big)^{1-\e/2}\right)+\frac{C}{N},
\end{align*}
which ends the proof.
\end{preuve}

\subsection{A second coupling}
Fournier  and Guillin \cite[Theorem 1]{MR3383341} proved that for a probability measure $f$ on $\mathbb{R}^d$ and a sequence of $f$-distributed and $\mathbb{R}^d$-valued i.i.d random variables $(X_1,\cdots,X_N)$, 
\[\E\cW_2^2\Big(f,\frac{1}{N}\sum_{i=1}^N\delta_{X_i}\Big)\le\frac{C_p(\int |v|^pf(dv))^{2/p}}{N^{1/2}},\]
for any $p>4$ with $C_p>0$. Compared to this,  in order to bound $\e^{K}_t(f_t^K)$, 
we thus have to prove the following result making use of the second coupling technique mentioned  previously because the family of processes  $(W_t^{i,K})_{t\ge0}$ are not independent.

\begin{prop}\label{thep}
For $\gamma\in(0,1]$ and taking $K\ge1$, suppose \eqref{cs}, \eqref{c4}, \eqref{c2hs} or \eqref{c2}. Let $f_0\in\cP_2(\rd)$ not be a Dirac mass and satisfy \eqref{c5}. Fix $N\geq 2$ and recall that $\e^{K}_t(f_t^K)$ defined in (\ref{bestrate}).
Consider the unique weak solution $(f_t^K)_{t\geq 0}$ to the cutoff \eqref{be} defined in Remark  \ref{cutws}. For any $0<\varepsilon<1,$  we have
\beq\label{spe}
\e^{K}_t(f_t^K)\le C_{\e,T}\Big(\frac{1}{N^\frac{1}{3}}\Big)^{1-\varepsilon}.
\eeq
\end{prop}

Before  we give the proof, we first check that the family of processes $(W_t^{i,K})_{t\ge0}$ without independence are  not so far from  independent in some sense.
\begin{lem}\label{thele}
For $k\in{\{1,...,N\}}$, there is an i.i.d. family of processes $(\tW_t^{i,K})_{i=1,...,k, t\ge0}$
such that $\tW_t^{i,K}\sim f_t^K$ for all $t\ge0$, all $i=1,...,k$ and such that for all $T>0$, all $\varepsilon\in(0,1)$ and all $i=1,...,k$, 
\beq\label{decoup}
\E[|W_t^{i,K}-\tW_t^{i,K}|^2]\le C_{\e,T} \big(\frac{k}{N}\big)^{1-\varepsilon}.
\eeq
\end{lem} 
\begin{preuve}. We will write our proof into two steps. Let  $k\in\{1,...,N\}$ be fixed.
\vip
{\it Step} 1.
Recall that the Poisson random measures $(M_{ij}(ds,d\alpha,dz,d\varphi))_{1\le i<j\le N}$ introduced  in Lemma \ref{kps} are independent, and satisfy $M_{ij}(ds,d\alpha,dz,d\varphi)=M_{ji}(ds,d\alpha,dz,d\varphi)$.
We introduce a new family of independent Poisson measures $(\tM_{ij}(ds,d\alpha,dz,d\varphi))_{1\le i\ne j\le N}$  on $[0,\infty)\times [0,1]\times[0,\infty)\times [0,2\pi)$ with intensity measure $\frac{1}{N-1}ds d\alpha dz d\varphi$ (independent of everything else), and $\tM_{ii}(ds,d\alpha,dz,d\varphi)=0$ for $i=1,...,N$.
We now define for $i=1,...,k$,
\[N_{ij}(dt,d\alpha,dz,d\varphi)=\tM_{ij}(dt,d\alpha,dz,d\varphi)\indiq_{1\le j\le k}+M_{ij}(dt,d\alpha,dz,d\varphi)\indiq_{j>k},\]
which are independent Poisson measures, with intensity $\frac{1}{N-1}ds d\alpha dz d\varphi$. It follows from Lemma \ref{coupling1} that for $i=1,...,k$, there is a process  $\tW_t^{i,K}\sim f_t^K$ for all $t\ge0$, starting from $V^i_0$ solving the  stochastic equation 
\begin{align*}
    \tW^{i,K}_t=V^i_0 + \sum_{j=1}^N\intot\int_0^1\int_0^\infty\int_0^{2\pi} 
c_K(\tW^{i,K}_\sm,\Pi_s^{ij,K}(\bW^K_\sm,\alpha),z,\varphi+\tilde{\varphi}_{i,j,s})N_{ij}(ds,d\alpha,dz,d\varphi).
\end{align*}
where $\tilde{\varphi}_{i,j,s}:=\varphi_{i,j,s}+\varphi_{i,j,s}^1$ with $\varphi_{i,j,s}:=\varphi_0(V^{i}_\sm-V^{j}_\sm,W^{i,K}_\sm-\Pi_s^{ij,K}(\bW^K_\sm,\alpha))$, and $\varphi_{i,j,s}^1:= \varphi_0(W^{i,K}_\sm-\Pi_s^{ij,K}(\bW^K_\sm,\alpha),\tW^{i,K}_\sm-\Pi_s^{ij,K}(\bW^K_\sm,\alpha)).$ And $(\tW^{i,K})_{i=1,...,k}$ are independent since the Poisson measures $N_{ij}$ are independent and independent of everything else.

\vip

{\it Step} 2. 
In this step, we show that for all $T>0$, all $\varepsilon\in(0,1)$ and all $i=1,...,k$, 
\beq\label{decoup}
\E[|W_t^{i,K}-\tW_t^{i,K}|^2]\le C_{\e,T} \big(\frac{k}{N}\big)^{1-\varepsilon}.
\eeq
By exchangeability, it suffices to study $\E[|W_t^{1,K}-\tW_t^{1,K}|^2]$. Recalling that $W_t^{i,K}$ defined in \eqref{tbp2} and $\tW_t^{i,K}$ introduced above, we have 
\begin{align*}
&W_t^{1,K}-\tW_t^{1,K}\\
&=\sum_{j=1}^k\intot\int_0^1\int_0^\infty\int_0^{2\pi} 
c^{1j}_K M_{1j}(ds,d\alpha,dz,d\varphi)-\sum_{j=1}^k\intot\int_0^1\int_0^\infty\int_0^{2\pi} 
\tilde{c}^{1j}_K \tM_{1j}(ds,d\alpha,dz,d\varphi)\\
&\hskip2cm+\sum_{j=k+1}^N\intot\int_0^1\int_0^\infty\int_0^{2\pi} 
(c^{1j}_K-\tilde{c}^{1j}_K) M_{1j}(ds,d\alpha,dz,d\varphi).
\end{align*}
where 
\begin{align*}
    c^{1j}_K:=c_K(W^{1,K}_\sm,\Pi_s^{1j,K}(\bW^K_\sm,\alpha),z,\varphi+\varphi_{1,j,s}),\  \tilde{c}^{1j}_K:=c(\tW^{1,K}_\sm,\Pi_s^{1j,K}(\bW^K_\sm,\alpha),z,\varphi+\tilde{\varphi}_{1,j,s}).
\end{align*}
By It\^o's formula, we have
\begin{align*}
    \E[|W_t^{1,K}-\tW_t^{1,K}|^2]=J^{1}_t+J^{2}_t+J^{3}_t,
\end{align*}
where
\begin{align*}
   & J^{1}_t=\frac{1}{N-1}\sum_{j=1}^k\intot\int_0^1\int_0^\infty\int_0^{2\pi} \E\Big[|W_s^{1,K}-\tW_s^{1,K}+c^{1j}_K|^2-|W_s^{1,K}-\tW_s^{1,K}|^2\Big]d\varphi dzd\alpha ds,\\
     &J^{2}_t=\frac{1}{N-1}\sum_{j=1}^k\intot\int_0^1\int_0^\infty\int_0^{2\pi} \E\Big[|W_s^{1,K}-\tW_s^{1,K}-\tilde{c}^{1j}_K|^2-|W_s^{1,K}-\tW_s^{1,K}|^2\Big]d\varphi dzd\alpha ds,\\
     &J^{3}_t=\frac{1}{N-1}\sum_{j=k+1}^N\intot\int_0^1\int_0^\infty\int_0^{2\pi}\E\Big[ |W_s^{1,K}-\tW_s^{1,K}+c^{1j}_K-\tilde{c}^{1j}_K|^2-|W_s^{1,K}-\tW_s^{1,K}|^2\Big]d\varphi dzd\alpha ds.
\end{align*}
It's not hard to bound $J^{1}_t$ and $J^{2}_t$. Applying  Lemma \ref{fundest}-(ii) with $v=W_s^{1,K}, v_*=\Pi_s^{1j,K}(\bW^K_\sm,\alpha), \tv=\tW_s^{1,K}$, we have from Lemma \ref{ww2p}-(ii) and (\ref{momex}) that 
\begin{align*}
J^{1}_t&\le \frac{C}{N-1}\sum_{j=1}^k\intot\int_0^1\E\Big[1+|\tW_s^{1,K}|^{2+2\gamma}+|W_s^{1,K}|^{2+2\gamma}+|\Pi_s^{1j,K}(\bW^K_\sm,\alpha)|^{2+2\gamma}\Big]d\alpha ds\\
&\le \frac{Ck}{N-1}\intot\Big[\int_\rd v^{2+2\gamma}f_s^K(dv)+1\Big]ds\\
&\le \frac{Ctk}{N-1}.
\end{align*}
Similarly, $J^{2}_t\le \frac{Ctk}{N-1}.$ We now bound  $ J^{3}_t$.
By \eqref{cesti} with $v_*=\tv_*=\Pi_s^{1j,K}(\bW^K_\sm,\alpha)$, we have
\begin{align*}
J^{3}_t\le \intot B(s)ds,
\end{align*}
where 
\begin{align*}
B(s)=\frac{1}{N-1}\int_0^1 \E\Big[\sum^N_{j=k+1}(I_1^K+I_2^K)(W^{1,K}_s, \Pi_s^{1j,K}(\bW^K_\sm,\alpha),\tW_s^{1,K},\Pi_s^{1j,K}(\bW^K_\sm,\alpha)) \Big] d\alpha.
\end{align*}
We deduce from Lemma \ref{IK247} that for  fixed $q\in(\gamma,p_0)$,
\begin{align*}
B(s)\le& \frac{CM}{N-1}\sum^N_{j=k+1}\int_0^1 \E\Big[|W^{1,K}_s-\tW_s^{1,K}|^2\Big] d\alpha\\
&+\frac{C_qe^{-M^{q/\gamma}}}{N-1}\sum^N_{j=k+1}\int_0^1 \E\Big[\exp\big(C_q(|W^{1,K}_s|^q+|\Pi_s^{1j,K}(\bW^K_\sm,\alpha)|^q+|\tW_s^{1,K}|^q)\big)\Big] d\alpha\\
\le& CM\E\Big[|W^{1,K}_s-\tW_s^{1,K}|^2\Big]+Ce^{-M^\frac{q}{\gamma}}.
\end{align*}
The last  step derives from a similar computation  to \eqref{nones}, Lemma \ref{ww2p}-(ii) and (\ref{momex}).
We thus have (noting   $\frac{N}{N-1}<2$)
\begin{align*}
    \E[|W_t^{1,K}-\tW_t^{1,K}|^2]\le \frac{Ctk}{N}+CM\intot \E[|W_s^{1,K}-\tW_s^{1,K}|^2]ds+Cte^{-M^\frac{q}{\gamma}}.
\end{align*}
By the Gr\"{o}nwall inequality, we observe that 
$$
\E[|W_t^{1,K}-\tW_t^{1,K}|^2]\le Cte^{CMt }\Big(\frac{k}{N}+e^{-M^\frac{q}{\gamma}}\Big)\le CTe^{CMT}\Big(\frac{k}{N}+e^{-M^\frac{q}{\gamma}}\Big).
$$
If $1\ge\frac{k}{N}\ge e^{-1}$,  we  then choose $M=1$. Otherwise, by  choosing 
 $M=(\ln\frac{N}{k})^\frac{\gamma}{q},$ we then conclude $\E[|W_t^{1}-\tW_t^{1}|^2]\le \frac{CTk}{N}e^{CT(\ln\frac{N}{k})^\frac{\gamma}{q}}$. Since $\gamma/q<1$, then we have for all $0<\varepsilon<1$, 
$(\ln\frac{N}{k})^\frac{\gamma}{q}\le C_\e+\frac{\varepsilon}{CT}\ln\frac{N}{k}$, which gives (\ref{decoup}).
\end{preuve}

In order to prove Proposition  \ref{thep},  we also need the following lemma.
\begin{lem}\label{line}
Let $m\ge2$, $\mu\in\cP_2(\rd)$. Let $W_1,...,W_N$ be a family of exchangeable $\rd$-valued random variables. Then for any $k\le m$, we set $\mu_k=k^{-1}\sum_{i=1}^k \delta_{W_i}$,
\[\E\left[\cW_2^2(\mu_m, \mu)\right]\le \E\left[\cW_2^2(\mu_k, \mu)\right]+\frac{l}{m}\Big(2m_2(\mu)+2\E[|W_1|^2]\Big),\]
where $r, l$  are the unique non-negative integers satisfying $m=kr+l$ and $l\le k-1$.
\end{lem}
\begin{preuve}.
For $k\in{\{1,...,m\}}$,  we observe that 
\[\mu_m=\frac{kr}{m}\mu_{kr}+\frac{l}{m}\cdot\frac{1}{l}\sum_{i=kr+1}^{m}\delta_{W_i}.\]First, by (\ref{wline}), we have 
\begin{align*}
\cW_2^2\Big(\mu_m, \mu\Big)\le \frac{kr}{m}\cW_2^2\Big(\mu_{kr}, \mu\Big)+\frac{l}{m}\cW_2^2\Big(\frac{1}{l}\sum_{i=kr+1}^{m}\delta_{W_i}, \mu\Big).
\end{align*}
Applying again  (\ref{wline}), we have 
\begin{align*}
\cW_2^2\Big(\frac{1}{kr}\sum_{i=1}^{kr}\delta_{W_i}, \mu\Big)\le \frac{1}{r}\cW_2^2\Big(\frac{1}{k}\sum_{i=1}^{k}\delta_{W_i}, \mu\Big)+ \frac{r-1}{r}\cW_2^2\Big(\frac{1}{k(r-1)}\sum_{i=k+1}^{kr}\delta_{W_i}, \mu\Big),
\end{align*}
 since $\frac{1}{kr}\sum_{i=1}^{kr}\delta_{W_i}=\frac{1}{r}\cdot \frac{1}{k}\sum_{i=1}^{k}\delta_{W_i}+(1-\frac{1}{r})\cdot\frac{1}{k(r-1)}\sum_{i=k+1}^{kr}\delta_{W_i}$.
Similarly, we finally have 
\[\cW_2^2\Big(\mu_m, \mu\Big)\le \frac{k}{m}\sum_{u=1}^r\cW_2^2\Big(\frac{1}{k}\sum_{i=k(u-1)+1}^{ku}\delta_{W_i}, \mu\Big)+\frac{l}{m}\cW_2^2\Big(\frac{1}{l}\sum_{i=kr+1}^{m}\delta_{W_i}, \mu\Big).\]
Using the exchangeability, we thus find that 
\[\E\left[\cW_2^2(\mu_m, \mu)\right]\le\frac{kr}{m}\E[\cW_2^2(\mu_k, \mu)]+\frac{l}{m}\E\left[\cW_2^2\Big(\frac{1}{l}\sum_{i=kr+1}^{m}\delta_{W_i}, \mu\Big)\right]. \]
This completes the proof since $l\le k-1$ and $\frac{kr}{m}\le 1$ and because 
$\E\left[\cW_2^2\Big(\frac{1}{l}\sum_{kr+1}^{m}\delta_{W_i}, \mu\Big)\right]\le 2m_2(\mu)+2\E[|W_1|^2].$

\end{preuve} 
We now move to verify Proposition \ref{thep}.
\begin{preuve} {\it of Proposition \ref{thep}.} 
Set $m=N-1$, $W_i=W_t^{i,K}$ for $i=1,...,N-1$ and $\mu=f_t^K$ in Lemma \ref{line} and given $k\le N-1$. We  recall $\e^{K}_t(f_t^K)= \E \left[\cW_2^2\left(f_t^K,\frac{1}{N-1}\sum_{i=2}^N \delta_{W^{i,K}_t} \right)\right]$ and  get by exchangeability and Lemma \ref{thele},
\begin{align*}
    \e^{K}_t(f_t^K)\le& \E\left[\cW_2^2\left(\frac{1}{k}\sum_{i=1}^{k}\delta_{W_t^{i,K}},f_t^K\right)\right]
    +\frac{k}{N-1}\Big(4m_2(f_t^K)\Big)\\
    \le & 2\E[|W^{1,K}_t-\tW^{1,K}_t|^2]+2\E\left[\cW_2^2\left(\frac{1}{k}\sum_{i=1}^{k}\delta_{\tW_t^{i,K}},f_t^K\right)\right]+8\frac{k}{N}.
\end{align*}
By \eqref{momex} in Theorem \ref{wp}, we know that $f_t^K\in\cP_p(\rd) $ with $p>4$, hence by Theorem 1 in \cite{MR3383341}, we have 
\[\E\left[\cW_2^2\left(\frac{1}{k}\sum_{i=1}^{k}\delta_{\tW_t^{i,K}},f_t^K\right)\right]\le Ck^{-1/2}.\] Using Lemma \ref{thele}, we finally conclude that 
\[ \e^{K}_t(f_t^K)\le C_{\e,T} \Big(\frac{k}{N}\Big)^{1-\varepsilon}+\frac{C}{\sqrt k}+8\frac{k}{N}.\]
Moreover, the result (\ref{spe}) yields by choosing $k=\lfloor N^{\frac{2}{3}}\rfloor$ for any $0<\varepsilon<1$.
\end{preuve}
Finally, we prove the main result.
\subsection{Proof of Theorem \ref{mr}}
From Proposition \ref{mrl} and  \ref{thep}, we already know for $0<\gamma<1,$ for any $0<\varepsilon<1/2,$ there exists $C_\varepsilon>0$ such that
\begin{align*}
   \sup_{[0,T]}\E[\cW_2^2(\mu^{N}_t,f_t)] \leq& C_{\e,T} \left(\sup_{[0,T]}\e^K_t(f_t^K)
+ K^{1-2/\nu} \right)^{1-\e}+\frac{C}{N}\\
&\le C_{\e,T} \left(\Big(N^{-1/3}\Big)^{1-\e}
+ K^{1-2/\nu} \right)^{1-\e}+\frac{C}{N}\\
&\le C_{\e,T} \left(N^{-1/3}
+ K^{1-2/\nu} \right)^{1-2\e},
\end{align*}
Letting $K$ go  to infinity, which finishs the proof of point (i).
Point (ii) is for $\gamma=1$,  the proof is exactly the same. Combining Proposition \ref{mrl} with \ref{thep}, we have 
\begin{align*}\label{fc51}
\sup_{[0,T]}\E[\cW_2^2(\mu^{N}_t,f_t)] \leq& C_{\e,T} \left( K^{2\e} (N^{-1/3})^{(1-\e)^2}
+ K^{1-2/\nu} \right)+\frac{C}{N} .
\end{align*}
By choosing $K=N^{(2/\nu-1)/3}$, we have for any $\e\in(0,\nu/4),$ 
\begin{align*} 
\sup_{[0,T]}\E[\cW_2^2(\mu^N_t,f_t)] &\leq C_{\e ,T}(N^{-1/3})^{(1-\e)^2}N^{(4/\nu-2)\e/3}\\
&=  C_{\e ,T}(N^{-1/3})^{1-4\e/\nu+\e^2}\\
&\le  C_{\e ,T}(N^{-1/3})^{1-4\e/\nu}.
\end{align*}
We conclude the result.

\section*{Acknowledgements}
\quad~~
We would like to thank greatly Nicolas Fournier for bringing us to this interesting problem and some helpful discussions. Liping Xu is partly supported by National Natural Science Foundation of China (12101028).

\bibliographystyle{abbrv} 
\bibliography{kac}

\end{document}